\documentclass[11pt, oneside]{amsart}
\usepackage{amsmath}
\usepackage{amssymb}
\usepackage{mathrsfs}
\usepackage{eucal}
\usepackage[all]{xy}
\usepackage{float}
\usepackage{stmaryrd}
\usepackage{todonotes}

\usepackage{enumitem}
\usepackage{eucal}
\usepackage[textwidth=14.5cm, textheight=22cm]{geometry}
\usepackage{setspace}
\usepackage[bookmarksnumbered]{hyperref}
\makeatletter
\def\paragraph{\@startsection{paragraph}{4}%
  \z@\z@{-\fontdimen2\font}%
  {\normalfont\bfseries}}
\makeatother

\DeclareMathOperator{\Sp}{Sp}
\DeclareMathOperator{\Top}{top}

\DeclareMathOperator{\pullback}{pb}
\DeclareMathOperator{\cell}{--cell--}

\DeclareMathOperator{\id}{Id}
\DeclareMathOperator{\h}{H}
\DeclareMathOperator{\B}{B}
\DeclareMathOperator{\ho}{Ho}
\DeclareMathOperator{\leftmod}{--mod}
\DeclareMathOperator{\ext}{Ext}
\DeclareMathOperator{\homSSS}{Hom}
\renewcommand{\hom}{\homSSS}
\DeclareMathOperator{\LimSSS}{lim}
\renewcommand{\lim}{\LimSSS}

\DeclareMathOperator{\colim}{colim}
\DeclareMathOperator{\ch}{Ch}

\newcommand{\sphspec}{{\mathbb{S}}}
\newcommand{\Id}{{\mathrm{Id}}}
\newcommand{\bQ}{{\mathbb{Q}}}
\newcommand{\bU}{{\mathbb{U}}}
\newcommand{\bP}{{\mathbb{P}}}
\newcommand{\cF}{\mathcal{F}}
\newcommand{\cE}{\mathcal{E}}
\newcommand{\cO}{\mathcal{O}}

\newcommand{\lra}{\longrightarrow}

\newcommand{\cA}{\mathcal{A}}
\newcommand{\efp}{E\mathcal{F}_+}
\newcommand{\etf}{\widetilde{E}\mathcal{F}}

\newcommand{\sm}{\wedge}

\newcommand{\Q}{\mathbb{Q}}

\newcommand{\torus}{{ \mathbb{T} }}
\newcommand{\algmodel}{\mathcal{A}(\mathbb{T})}
\newcommand{\diffalgmodel}{d \mathcal{A}(\mathbb{T})}

\newcommand{\TSP}{{\torus \Sp^\ocal}}

\newcommand{\adjunct}{{ \,\,\raisebox{-0.1\height}{$\overrightarrow{\longleftarrow}$}\,\, }}

\newcommand{\smashprod}{\wedge}
\newcommand{\co}{\colon \!}

\newcommand{\fibrep}{{ \widehat{f} }}
\newcommand{\cofrep}{{ \widehat{c} }}

\newcommand{\zz}{{ \mathbb{Z} }}

\newcommand{\qq}{{ \mathbb{Q} }}

\newcommand{\cscr}{{ \mathscr{C} }}

\newcommand{\pscr}{{ \mathscr{P} }}

\newcommand{\acal}{{ \mathcal{A} }}
\newcommand{\bcal}{{ \mathcal{B} }}
\newcommand{\ccal}{{ \mathcal{C} }}
\newcommand{\dcal}{{ \mathcal{D} }}
\newcommand{\ecal}{{ \mathcal{E} }}
\newcommand{\fcal}{{ \mathcal{F} }}

\newcommand{\ocal}{{ \mathcal{O} }}
\newcommand{\pcal}{{ \mathcal{P} }}

\newcommand{\ucal}{{ \mathcal{U} }}

\newtheorem{theorem}{Theorem}[subsection]
\newtheorem{proposition}[theorem]{Proposition}
\newtheorem{corollary}[theorem]{Corollary}
\newtheorem{lemma}[theorem]{Lemma}

\newtheorem*{thm}{Theorem}

\newtheorem{definition}[theorem]{Definition}

\newtheorem{rmk}[theorem]{Remark}

\begin{document}

\title{Rational $SO(2)$--Equivariant Spectra}
\author{D. Barnes}
\address{Pure Mathematics Research Centre, 
Queen's University Belfast, University Road, Belfast BT7 1NN, Northern Ireland, UK}
\email{d.barnes@qub.ac.uk} 
\author{J.P.C. Greenlees}
\address{School of Mathematics and Statistics, Hicks Building, 
Sheffield S3 7RH. UK.}
\email{j.greenlees@sheffield.ac.uk}
\author{ M. K\c{e}dziorek}
\address{MATHGEOM \\
  \'Ecole Polytechnique F\'ed\'erale de Lausanne \\
 CH-1015 Lausanne \\
 Switzerland}
\email{magdalena.kedziorek@epfl.ch} 
\author{B. Shipley}
\address{Department of Mathematics, Statistics, and Computer Science, University of Illinois at
Chicago, 508 SEO m/c 249,
851 S. Morgan Street,
Chicago, IL, 60607-7045, USA}
  \email{bshipley@math.uic.edu}

\date{}

\begin{abstract}
\noindent
We prove that the category of rational $SO(2)$--equivariant spectra has a simple
algebraic model. Furthermore, all of our model categories and Quillen equivalences are monoidal,
so we can use this classification to understand
ring spectra and module spectra via the algebraic model.
\end{abstract}


\maketitle

\setcounter{tocdepth}{2}
\pdfbookmark[1]{Contents}{toc}
\tableofcontents

\section{Introduction}
\paragraph{Rational equivariant cohomology theories}
This paper is a contribution to the study of equivariant cohomology theories, and gives a rather complete analysis for one class of theories.
To start with, $G$-equivariant cohomology theories are represented by $G$-spectra, so that the category of $G$-equivariant cohomology theories and
stable natural transformations between them is equivalent to the homotopy category of $G$-spectra, and it is natural to study the homotopy theory of
$G$-spectra. One cannot expect a complete analysis of either cohomology theories or spectra integrally, but if we rationalize,  things are greatly simplified,
whilst valuable geometric and group theoretic structures
remain. Henceforth we restrict attention to rational cohomology
theories and rational spectra without further comment. The general conjecture is
that there is a nice algebraic model for rational $G$-spectra. More precisely that there is an abelian category $\cA (G)$ and a  Quillen equivalence
$$\mbox{$G$-spectra}\simeq d\cA(G)$$
where $d\cA(G)$ consists of differential graded objects of $\cA(G)$. The category $\cA (G)$ is of injective dimension equal to the rank of $G$ and
of a form that is easy to use in calculations. Of course one would like to prove the Quillen equivalence reflects as much structure as possible. The conjecture
is known for quite a number of groups in some form, and we refer
to \cite{tnqcore} for a summary of what is known. In the present paper we are  concerned with the specific case of the circle group, and with giving a zig-zag of
Quillen equivalences which are symmetric monoidal.

\bigskip
\paragraph{The circle group}
We will entirely focus on the circle group, because it plays a critical role in understanding the case of all other infinite compact Lie groups. As an added benefit,
 it is significantly simpler than any other group, meaning that we can focus on the critical features without being distracted by extraneous complication.
 We refer to the group as $SO(2)$, because we have in mind as first
 applications its role as a subgroup of $O(2)$ (in \cite{barneso2})
 and $SO(3)$  (in  \cite{KedziorekSO(3)}).

Our main  result is as follows.

\begin{thm}
The model category of rational $SO(2)$--equivariant spectra is Quillen
equivalent to the algebraic model
$d \acal(SO(2))_{dual}$.
Furthermore these Quillen equivalences are all symmetric monoidal,
hence the homotopy category of rational $SO(2)$--equivariant spectra
and the homotopy category of the algebraic model $D( \acal(SO(2)))$ 
are equivalent as symmetric monoidal categories.
\end{thm}

The algebraic model is described in Section \ref{sec:cyclicmodel}  below.

\bigskip
\paragraph{Rings and commutative rings}
Our main theorem establishes a zig-zag of symmetric monoidal Quillen equivalences
between the symmetric monoidal model category of rational $SO(2)$-spectra
and the symmetric monoidal model category $d\cA (SO(2))_{dual}$.
In particular
we may use \cite[Theorem 3.12]{ss03monequiv} to see that the model category of ring spectra is Quillen equivalent to the category of monoids in $\cA (SO(2))$.
This means that a ring object $R_a$ in $d\cA (SO(2))$ corresponds to a  ring object $R_{top}$ in $SO(2)$-spectra in a homotopy invariant fashion.
Furthermore, the category of $R_a$-modules is Quillen equivalent to the category of $R_{top}$-modules.

However, it is essential to emphasize that if $R_a$ is commutative, it does not follow that $R_{top}$ has a commutative model. The reason is that
the correspondence between $R_a$ and $R_{top}$ involves fibrant and cofibrant approximations and  these approximations are
only in the category of rings rather than in the category of commutative rings. This is inevitable, since for example the ring spectrum $R_{top}=\etf$ corresponds
to a small and explicit commutative ring $R_a$. It is well known
\cite{McClureTate}
that $\etf$ is not a commutative ring in orthogonal $SO(2)$-spectra.

In \cite{gre05} it is shown that if $C$ is a generalized elliptic curve over a $\Q$-algebra, there is an associated $SO(2)$-spectrum $EC$ representing
elliptic cohomology. Indeed the proof proceeds by writing down an object $EC_a$ in $\cA (SO(2))$, and taking $EC=EC_{top}$ to be the corresponding $SO(2)$-spectrum.
It is transparent from the construction that $EC_a$ is a commutative ring in $\cA (SO(2))$, and it is a consequence of the present work that $EC$ is a ring
spectrum. As commented above, this does not prove that $EC$ is a commutative ring spectrum.

\bigskip
\paragraph{Contribution of this paper}

To place the contribution of this paper to the study of rational $SO(2)$-spectra, we need to give a little history.
A description of the homotopy category of rational $SO(2)$-spectra
was given in \cite{gre99}. This took the form of an equivalence
\[\ho(\mbox{$SO(2)$-spectra})\simeq D(\cA (SO(2)))\]
for the abelian category $\cA (SO(2)$ (described in Section \ref{sec:cyclicmodel} below). Since $\cA (SO(2))$ is rather simple and of injective dimension 1
 this gives a practical means for calculating the space $[X,Y]^{SO(2)}_*$ of maps
for arbitrary (rational) $SO(2)$-spectra $X,Y$ up to extension. Since $\cA (SO(2))$ is (in a sense that will appear later) evenly graded, the extensions split, and
so \cite{gre99} gives a complete description of the category $\ho(\mbox{$SO(2)$-spectra}$. Unfortunately, \cite{gre99} claimed to have proved this was an
equivalence of triangulated categories, but there is a gap in the
argument for this (this was pointed out by Patchkoria, who has now
given (\cite{Patchkoria12} and more recent work) an illuminating
systematic analysis of lifting equivalences of homotopy categories to
ones that preserve triangulations and other structures, giving
conditions in terms of injective dimension and sparsity; the argument
for $\cA (SO(2))$ in \cite{gre99} is of this type, but outside the
Patchkoria range). The article \cite{shi02} showed that in this case a triangulated
equivalence of homotopy categories lifts to a Quillen equivalence of
model categories. Work then began to give an algebraic model for the
homotopy category of $G$-spectra for a torus $G $ (eventually leading
to \cite{tnqcore}); it was soon apparent that the only way to approach this is to first prove a Quillen equivalence between $G$-spectra and $d\cA (G)$ and then deduce the equivalence of homotopy categories as a consequence.  This general project has taken some time, and has a complicated history of its own \cite{gscell, gsmodules, gsfixed, tnqcore} but the special case of the circle is much simpler than the general case, and quite easily explained. The underlying strategy applied in \cite{tnqcore} is the same as that  adopted here for the circle group, but there are some significant differences of implementation adopted from \cite{barnesthesis, barneso2, kedziorekthesis, KedziorekSO(3)}.

Meanwhile, work began on the group $O(2)$ (culminating in the model \cite{barneso2})  and the group $SO(3)$ (culminating in the model \cite{kedziorekthesis}).
Those models depended  on the Quillen equivalence for $SO(2)$; they originally built upon \cite{tnqcore}, but the technical context adopted here has advantages for them. 
The proof for the general torus is considerably more complicated than that for the circle,  principally because
$SO(2)$ has only two connected subgroups (namely the trivial group and the whole group) rather than infinitely many for higher dimensional tori. Accordingly, it is
much easier to see the essential structure  of the argument in the case of the circle. It is therefore desirable to give a separate account for $SO(2)$ to show the simplicity
of the argument, and  to provide the input to the work on $O(2)$ and $SO(3)$.

Perhaps a more important reason for publishing a separate account for
$SO(2)$ is that at present we can prove more for the circle group than
for a general torus.
The category of $G$-spectra is a monoidal model category, and  $\cA (G)$ is a monoidal abelian category. One would like to have a monoidal equivalence between $G$-spectra and $d\cA (G)$.
Of course this requires more care than a simple Quillen equivalence and some more delicate analysis. As the first step, one needs a monoidal model structure on $d\cA (G)$. The abelian category $\cA (G)$ does not have enough projectives and the  model structure on $d\cA (G)$ used in earlier work, i.e. the injective model structure, is certainly not monoidal. On the other hand for $G=SO(2)$,   \cite{barnesmonoidal} has given a model structure based on dualizable objects which is monoidal;
this relies on some explicit constructions in $\cA (SO(2))$ from \cite{gre99} that are not made explicit in \cite{gretnq1, gre12standardalgebra} for higher tori. It is expected that a similar construction will work for other groups, but additional work will be necessary. Once a monoidal model structure is defined on $d\cA (G)$, one needs to ensure that
all Quillen pairs making up  the equivalence are monoidal.

\bigskip
\paragraph{The Hasse-Tate isotropy square}
The overarching strategy for building an algebraic model  is to break
the category of $SO(2)$-spectra into parts, give algebraic models of
each part, and then assemble an algebraic model for all spectra from
the algebraic models of the parts in the corresponding way.

To analyse an individual $SO(2)$-spectrum it is natural to use isotropy separation, to assemble the spectrum from information at the family $\cF$ of
finite subgroups and the information at $SO(2)$ itself.This can be implemented using the Tate square
\[
\xymatrix@R=2pc @C=2pc{
X \ar[r] \ar[d]&X\sm \etf \ar[d]\\
F(\efp, X) \ar[r] & F(\efp, X)\sm \etf
}
\]
which expresses $X$ as the homotopy pullback of its $\cF$-completion,  $F(\efp, X)$,  and its localization away from $\cF$, $X\sm \etf$,
over the Tate object,  $F(\efp , X)\sm \etf$. Thus $X$ is the homotopy pullback of a  punctured square diagram (i.e., of shape $\bullet \longrightarrow \bullet \longleftarrow \bullet$).
The basic idea is to do this at the level of model categories. We would like to assemble the category of all $SO(2)$-spectra from
the category of $\cF$-complete objects and objects localized away from $\cF$. The way we do it here is to take suitable model categories of
$\cF$-complete spectra, of spectra away from $\cF$ and Tate spectra
and then construct a model structure on the category $S^{\bullet} \leftmod$
of diagrams of such objects: a cellularization
($K_{top} \cell S^{\bullet} \leftmod$)  of this model category
of punctured-square diagrams is then shown to be Quillen equivalent to
the original category of $SO(2)$-spectra
essentially using the fact that the Tate square is a homotopy pullback.  The machinery of \cite{gsmodules} was built for this purpose.

The alternative adopted in \cite{tnqcore} is to say that the category of $SO(2)$-spectra is equivalent to the category of $S$-modules in $SO(2)$-spectra, where $S$ is the sphere spectrum.
We then consider the special case of the Tate square in which $X=S$ and say that $S$ is the pullback of a diagram of rings, so that the module category
of $S$ is Quillen equivalent to a cellularization of the model category of modules over the pullback diagram of rings.

In the present paper, several of the monoidal functors are taken from \cite{kedziorekthesis, KedziorekSO(3), barnesthesis, barneso2} and since  we work in a context where $\etf$ is not a commutative ring we adopt their methods for the formality argument.

\bigskip
\paragraph{Summary of the zig-zag of Quillen equivalences}

To illustrate the zig-zag of Quillen equivalences we present a diagram
which shows main steps of this comparison. The first step moving into
categories of diagrams was suggested in the previous subsection, and
the other steps will be described in the body of the paper. 
The reader may wish to refer to this diagram now, but the notation will be introduced as we proceed.
In the diagram left Quillen functors are placed on the left
and $\mathbb{T}:=SO(2)$. References to specific results are given on the left, and on the right there is an indication of the ambient category.

\[
\xymatrix@R=0.4pc @C=7pc{
 & L_{S_\bQ}\TSP
\ar@<-1ex>[dd]_{S^\bullet \smashprod - } & **[l] \mathrm{in\ } \TSP 
\\ **[r] \mathrm{Proposition\ } \ref{prop:Smodulesequivalence} \\
 & K_{\Top} \cell S^\bullet \leftmod
\ar@<-1ex>[uu]_{\mathrm{pb}}
\ar@<+1ex>[dd]^{(-)^{\torus}} & **[l] \mathrm{in\ } \TSP
\\ **[r] \mathrm{Corollary\ } \ref{cor:removeequivariancecell} \\
& K_{\Top}^{\torus} \cell S_{\Top}^\bullet \leftmod
\ar@<+1ex>[uu]^{a_{\#}}
\ar@<+0ex>[dd]^{\mathrm{of\ Quillen\ equivalences}} & **[l] \mathrm{in\ Sp}^\ocal
\\ **[r] \mathrm{Corollary\ } \ref{cor:gettingtoalgebra} \\
& K_t \cell S_t^\bullet \leftmod
\ar@<+0ex>[uu]^{\mathrm{zig-zag}}
\ar@<+0ex>[dd]^{\mathrm{of\ Quillen\ equivalences}} & **[l] \mathrm{in\ Sp}^\ocal
\\ **[r] \mathrm{Section\ } \ref{subsec:ramod} \\
& K_a  \cell S_a^\bullet \leftmod
\ar@<+0ex>[uu]^{\mathrm{zig-zag}}
\ar@<+1ex>[dd]^{\Gamma} & **[l] \mathrm{in\ Ch}_\bQ
\\ **[r] \mathrm{Proposition\ } \ref{prop:conclusion} \\
& d \acal_{dual}
\ar@<+1ex>[uu]^{l^*} & **[l] \mathrm{in\ Ch}_\bQ
}
\]

\bigskip
\paragraph{Notation}
From now on we will write $\torus$ for the group $SO(2)$. We also stick to the convention of drawing the left adjoint above the right one in any adjoint pair.

\section{The algebraic model
\texorpdfstring{$d\acal(\torus)$}{dA(SO(2))}}\label{sec:cyclicmodel}

In this section we recall the algebraic category $\algmodel$ as developed by 
the second author in \cite{gre99}. 
This category is naturally enriched in graded abelian groups. 
We use the notation $\diffalgmodel$ for the category of objects in 
$\algmodel$ with a differential
and call it the algebraic model for rational $\torus$--spectra.
A non-monoidal model structure for the category
$\diffalgmodel$ is given in \cite{gre99}.
Work of the first author \cite{barnesmonoidal} builds upon this and constructs a
monoidal model structure on $\diffalgmodel$.

In this paper, we call $\acal(\torus)$ the \textbf{abelian model} for rational $\torus$-spectra
and $\diffalgmodel$ the \textbf{algebraic model} for rational $\torus$-spectra.
The model structures we construct on $\diffalgmodel$ are such that
$\ho (\diffalgmodel)$ is equivalent to the derived category of the abelian model, $D(\algmodel)$, 
which is equivalent to the homotopy category of rational $\torus$-spectra by \cite{gre99}.

\subsection{The abelian model \texorpdfstring{$\acal(\torus)$}{A(T)}}\label{subsec:the model}
The abelian model for rational $\torus$--spectra
is established in \cite{gre99}.
We introduce this category, explain how to
turn it into a differential graded category and then define the
injective model structure.

\begin{definition}
Let $\fcal$ be the set of finite subgroups of $\torus$.
Let $\ocal_\fcal$ be the graded ring of operations $\prod_{n \geqslant 1} \qq [c_n]$
with $c_n$ of degree $-2$. Let $e_n$ be the idempotent arising from
projection onto factor $n$.
In general, let $\phi$ be a subset of $\fcal$ and
define $e_\phi$ to be the idempotent
coming from projection onto the factors in $\phi$.
We let $c$ be the unique element of $\ocal_\fcal$ such that $c_n = e_n c$ for all
$n \geq 1$.
\end{definition}

We use the notation 
$\ecal^{-1} \ocal_\fcal = \colim_{n \geqslant 1}  \ocal_\fcal[c_1^{-1}, \dots, c_n^{-1}]$. 
It is easy to see that $\ecal^{-1} \ocal_\fcal$ is a ring. 
The notation arises since this ring can also be described in terms of inverting 
a certain set of Euler classes.
As a vector space, $(\ecal^{-1} \ocal_\fcal)_{2k}$ is
$\prod_{n \geqslant 1} \qq $
for $k \leqslant 0$ and is $\oplus_{n \geqslant 1} \qq$ for
$n >0$.

For any $\ocal_\fcal$ module $N$, we define $\ecal^{-1} N$ to be
$\ecal^{-1} \ocal_\fcal \otimes_{\ocal_\fcal} N$.

\begin{definition}
We define the \textbf{abelian model}  $\acal=\acal(\torus)$ as follows.
Its class of objects is the
collection of triples $(N, U, \beta)$ where $N$ is an $\ocal_\fcal$--module, $U$ is a graded rational vector space and
\[
\beta \co N \lra \ecal^{-1} \ocal_\fcal \otimes U
\]
is an $\ocal_\fcal$--module map such that $\ecal^{-1} \beta$ is an isomorphism.\footnote{The tensor product in the target of $\beta$ is over $\qq$, thus we omit it from the notation.}

A map $(\theta, \phi)$ in $\acal$ is a commutative square
$$
\xymatrix@C+1cm{
N \ar[r]^(0.4)\beta \ar[d]_\theta &
\ecal^{-1} \ocal_\fcal \otimes U \ar[d]^{\id \otimes \phi} \\
N' \ar[r]^(0.4){\beta'} &
\ecal^{-1} \ocal_\fcal \otimes U'
}$$
where $\theta$ is a map of $\ocal_\fcal$--modules and
$\phi$ is a map of graded rational vector spaces.
\end{definition}

The relation between this category and rational $\torus$--equivariant
spectra is given by the following pair of theorems from \cite{gre99}.

\begin{theorem}
The homotopy category of rational $\torus$--equivariant spectra
is equivalent to the derived category of $\acal$.
\end{theorem}

For a rational $\torus$--equivariant spectrum $X$, let $\pi^\acal_*(X)$
be the following object of $\acal$.
For details of the spectra $DE \fcal_+$ and $\widetilde{E} \fcal$
see Definition \ref{def:universalspaces}. The spectrum $\Phi^{\torus} X$
is the geometric $\torus$--fixed points of $X$.
\[
\pi^\acal_*(X) = \big(
\pi_*^{\torus} (X \smashprod DE \fcal_+)
\longrightarrow
\pi_*^{\torus} (X \smashprod DE \fcal_+ \smashprod \widetilde{E} \fcal)
\cong \ecal^{-1} \ocal_\fcal \otimes \pi_*(\Phi^{\torus} X) \big)
\]
There is also an Adams short exact sequence which explains
how to calculate maps in the homotopy category of
rational $\torus$--equivariant spectra.

\begin{theorem}\label{thm:torusadams}
Let $X$ and $Y$ be rational $\torus$--equivariant spectra. Then
the sequence below is exact.
\[
0 \lra
\ext_\acal (\pi_*^\acal(\Sigma X),\pi_*^\acal( Y))
\lra
[X,Y]^{\torus}_* 
\lra
\hom_\acal (\pi_*^\acal(X),\pi_*^\acal(Y))
\lra 0
\]
\end{theorem}

In \cite{gre99} a model structure is given for the category of objects in
$\acal$ that have a differential. We define what it means to
have a differential and then introduce the model structure. We will leave the proof that
$\acal$ has all small limits and colimits to the next subsection (see also \cite{gre99}).

If we think of $\ocal_\fcal$ as an object of $\ch(\qq)$ with trivial
differential, then we can consider the category of
$\ocal_\fcal$--modules in $\ch(\qq)$. Such an object $N$ is an
$\ocal_\fcal$--module in graded vector spaces along with maps
$d_n \co N_n \lra N_{n-1}$. These maps
satisfy the relations below.
\[
d_{n-1} \circ d_n =0
\quad
c d_n = d_{n-2} c
\]

\begin{definition}
We define the category  $d \acal= d \acal(\torus)$ as follows.
Its class of objects is the
collection of triples $(N, U, \beta)$ where $N$ is a rational chain complex with an action of $\ocal_\fcal$, $U$ is a rational chain complex and
\[
\beta \co N \lra \ecal^{-1} \ocal_\fcal \otimes U
\]
 is a $\ocal_\fcal$--module map in $\ch(\qq)$ such that $\ecal^{-1} \beta$ is an isomorphism.

A map $(\theta, \phi)$ in $d \acal$ is then a commutative
square as for $\acal$, such that
$\theta$ is a map in the category of
$\ocal_\fcal$--modules in $\ch(\qq)$
and $\phi$ is a map of $\ch(\qq)$.

We call this category the \textbf{algebraic
model for rational $\torus$--spectra}.
\end{definition}

For $A$ and $B$ in $d \acal$, we define
$\acal(A, B)_*$ to be the graded set of maps
from the underlying object of $A$ in $\acal$
to the underlying object of $B$ in $\acal$.
We equip this graded $\qq$--module with the
differential induced by the convention
$df_n = d_B f_n + (-1)^{n+1} f_n d_A$.
By considering an object of $\acal$ as an object of
$d \acal$ with no differential, we can restrict the definition
of $\acal(A,B)_*$ to allow for the case where $A$
is in $\acal$.

The following result is the subject of \cite[Appendix B]{gre99}.

\begin{proposition}
The category $d \acal$ has a model structure where the
class of weak equivalences is exactly the class of quasi--isomorphisms.
The class of cofibrations is the class of monomorphisms.
This is called the \textbf{injective model structure}.
We write
$d \acal_{i}$
to denote this model structure.
\end{proposition}

As we shall see shortly, the category $\acal$ has a monoidal product
which induces a monoidal product on $d \acal$.
But the injective model structure does not make
$d \acal$ into a monoidal model category.
This failure occurs because of $c$--torsion, just as
the injective model structure on $\ch(\zz)$ is not monoidal due to torsion.

This is a serious defect, as we are unable to effectively compare
this monoidal product to the smash product of $\torus$--spectra.
This defect is further complicated by the lack of projective objects of $\acal$.
There is however a cofibrantly generated monoidal model structure
on $d \acal$ which is Quillen equivalent to the
injective model structure. It is constructed in \cite{barnesmonoidal} and we recall it in the next section.

%%%%%%%%%%%%%%%%%%%%%%%%%%%%%%%%%%%%
\subsection{The monoidal model structure}\label{subsec:adjunctions} \label{subsec:dualizable}

This subsection has three aims, namely to prove that
$\acal$ and $d \acal$ have all small limits and colimits (see also \cite{gre99})
define the monoidal product and recall the dualizable model structure on $d\acal$ (see \cite{barnesmonoidal}) which is monoidal.
To do so, we will need to relate $\acal$
to a larger category which we introduce below.

We let $\hat{\acal}$ be category of triples $(N, U,
\alpha \co N \lra \ecal^{-1} \ocal_\fcal \otimes U)$
where $N$ is an $\ocal_\fcal$--module,
$U$ is graded $\qq$--module and  the map $\alpha$
is a map of $\ocal_\fcal$--modules.
A map of such diagrams is then a commutative diagram
as below where $\theta$ is a map of $\ocal_\fcal$--modules,
and $\phi$ is a map of graded $\qq$--modules.
\[
\xymatrix{
N \ar[r] \ar[d]^\theta &
\ecal^{-1} \ocal_\fcal \otimes U
\ar[d]^{\id \otimes \phi} \\
N' \ar[r] &
\ecal^{-1} \ocal_\fcal \otimes U'
}
\]
Thus $\hat{\acal}$ is $\acal$ without the restriction
that the structure map of an object should be isomorphism after $\ecal$ is inverted.
There is an adjunction
\[
\xymatrix{
j^* : \acal \ar@<+1ex>[r] &
\hat{\acal} : \Gamma_h \ar@<+0.5ex>[l]}
\]
where $j^*$ is the inclusion. The functor $j^*$ is full and faithful.
The explicit construction of the right adjoint $\Gamma_h$
is quite intricate and therefore we leave the
details to \cite[Section 20.2]{gre99}.

Our first use of the torsion functor $\Gamma_h$ is to define limits in $\acal$. It is
clear that the adjunction
$(j, \Gamma_h)$ passes to categories with differentials,
as does the following definition.
\begin{definition}
Let $I$ be some small category and let
$\{ N_i \lra \ecal^{-1} \ocal_\fcal \otimes U_i \}$
be the objects of some $I$--shaped diagram in $\acal$.
The colimit over $I$ is
\[
\colim_i N_i \lra \ecal^{-1} \ocal_\fcal \otimes (\colim_i U_i) .
\]
The limit is formed by first applying the functor $j^*$, taking limits in
$\hat{\acal}$ and then applying $\Gamma_h$.
In more detail, we construct the following pullback square.
\[
\xymatrix{
M \ar[r] \ar[d]_f &
\lim(N_i) \ar[d] \\
\ecal^{-1} \ocal_\fcal \otimes \lim(U_i) \ar[r] &
\lim(\ecal^{-1} \ocal_\fcal \otimes U_i)
}
\]
The limit of the $I$--shaped diagram
$\{ N_i \lra \ecal^{-1} \ocal_\fcal \otimes U_i \}$
is $\Gamma_h f$.
\end{definition}

Now we turn to the monoidal product of $\acal$ and $d \acal$.

\begin{definition}
For
$\beta \co N \lra \ecal^{-1} \ocal_\fcal \otimes U$ and
$\beta' \co N' \lra \ecal^{-1} \ocal_\fcal \otimes U'$ in $d \acal$, their
\textbf{tensor product} is
\[
\beta \otimes \beta' \co
N \otimes_{\ocal_\fcal} N'
\lra
(\ecal^{-1} \ocal_\fcal \otimes U)
\otimes_{\ocal_\fcal}
(\ecal^{-1} \ocal_\fcal \otimes U')
\cong
\ecal^{-1} \ocal_\fcal \otimes (U \otimes_\qq U')
\]
The unit of this monoidal product is the object
$S^0 = (i \co \ocal_\fcal \lra \ecal^{-1} \ocal_\fcal \otimes \qq$).
\end{definition}

This monoidal product is related to the smash product of spectra 
as we can see from the short exact sequence of \cite{gre99}
\[
0
\lra
\pi_*^\acal(X) \otimes \pi_*^\acal(Y)
\lra \pi_*^\acal(X \smashprod Y)
\lra
\Sigma \text{Tor} (\pi_*^\acal(X),\pi_*^\acal(Y))
\lra 0
\]

This monoidal structure is closed, that is, there is an internal
function object describing the $d \acal$--object of maps
between two objects. This functor is more complicated than the tensor
product and requires use of the torsion functor $\Gamma_h$.

\begin{definition}\label{def:function}
Consider two elements of $d \acal$,
\[
A= (\beta \co N \lra \ecal^{-1} \ocal_\fcal \otimes U)
\quad  {and} \quad
B= (\beta' \co N' \lra \ecal^{-1} \ocal_\fcal \otimes U')
\]
The
\textbf{function object} $F(A,B)$ is the map
$\Gamma_h \delta$, where $\delta$ is defined by the pullback square below.
\[
\xymatrix{
Q
\ar[dd]
\ar[r]^(0.4)\delta &
\ecal^{-1} \ocal_\fcal \otimes \hom_{\qq} (U,U')
\ar[d] \\
&
\hom_{\ocal_\fcal} (
\ecal^{-1} \ocal_\fcal \otimes U,
\ecal^{-1} \ocal_\fcal \otimes  U')
\ar[d] \\
\hom_{\ocal_\fcal} (N,N')
\ar[r] &
\hom_{\ocal_\fcal} (N,
\ecal^{-1} \ocal_\fcal \otimes  U')
}
\]
\end{definition}

The monoidal product and function object are related by a natural isomorphism.
Let $A$, $B$ and $C$ denote objects of $d \acal$,  then
\[
d \acal(A \otimes B, C) \cong d \acal(A, F(B,C))
\]

\begin{definition}\label{def:modulefunctors}
For $K \in \ch(\qq)$ we define $LK \in d \acal$ as
\[
LK= (i \otimes \id_K \co
\ocal_\fcal \otimes K \lra
\ecal^{-1} \ocal_\fcal \otimes K)
\]

For $A$ and $B$ in $d \acal$, we define
$\acal(A, B)_*$ to be the graded set of maps
of $\acal$ (ignoring the differential). We then equip this
graded $\qq$--module with the differential induced by the convention
$df_n = d_B f_n + (-1)^{n+1} f_n d_A$.
This construction gives a functor
\[
R \co d \acal \lra \ch(\qq) \quad
RA := \acal(S^0, A)_*
\]
\end{definition}

The functors $L$ and $R$ form an adjoint pair between
$\ch(\qq)$ and $d \acal$. Furthermore they
give $d \acal$
the structure of a closed
$\ch(\qq)$--module in the sense of \cite[Section 4.1]{hov99}.

This module structure and the closed monoidal product interact to give
$d \acal$  a tensor product, a cotensor product and
an enrichment over $\ch(\qq)$.
Let $K \in \ch(\qq)$
and $A = (\beta \co N \lra \ecal^{-1} \ocal_\fcal \otimes U)$ in
$d \acal$.
Their \textbf{tensor product} $A \otimes K$ is defined to be  $A \otimes LK$.
Thus $A \otimes K$ is given by
\[
\beta \otimes \id_K  \co N \otimes_\qq K \lra
\ecal^{-1} \ocal_\fcal \otimes (U \otimes_\qq K)
\]
There is a \textbf{cotensor product} $A^K$ defined to be
$F(LK, A)$.
The \textbf{enrichment} is given by  $R F(A,B)$
for $A$ and $B$ in $d \acal$.
This enrichment, tensor and cotensor are related by the natural
isomorphisms below.
\[
d \acal(A, B^K)
\cong
d \acal(A \otimes K, B)
=
d \acal(A \otimes LK, B)
\cong
\ch(\qq)(K, R F(A,B))
\]

Now we are ready to recall the monoidal model structure on
$d \acal$ from \cite{barnesmonoidal} and compare it to several other model categories, in particular the injective model structure on $d \acal$ introduced in \cite{gre99}.
This monoidal model structure is defined in
terms of the (strongly) dualizable objects of $d \acal$.

\begin{definition}
An object $A \in \acal$ is said to be (strongly) \textbf{dualizable}
if for any $B \in \acal$ the canonical map
\[
F(A,S^0) \otimes B \longrightarrow F(A,B)
\]
is an isomorphism.
The \textbf{functional dual} of an object $B$ is the object $DB = F(B,S^0)$.
\end{definition}

Let $\pcal$ denote a set of representatives for the
isomorphisms classes of dualizable objects in $\acal$.
Such a set exists by \cite[Corollary 5.8]{barnesmonoidal}.
The following theorem summarizes \cite[Section 6]{barnesmonoidal}.

\begin{theorem}\label{thm:modelstructure}
There is a cofibrantly generated model structure
on $d \acal$ with weak equivalences the
class generated by the homology isomorphisms. The generating cofibrations
have the form
\[
S^{n-1} \otimes P \lra D^n \otimes P
\]
for $P \in \pcal$ and $n \in \zz$, where $S^n$ is the chain complex consisting of one copy of $\mathbb{Q}$ in degree $n$ and $0$ elsewhere and $D^n$ consists of two copies of $\mathbb{Q}$ in degrees $n$ and $n-1$ with the identity as the only non-trivial differential.

We call this model structure the \textbf{dualizable model structure}
and denote it $d \acal_{dual}$.
The dualizable model structure is proper,
symmetric monoidal and satisfies the monoid axiom.
\end{theorem}

Since all dualizable cofibrations are in particular monomorphisms we get the following comparison with the injective model structure of \cite{gre99}, which we write
as $d \acal_i$.

\begin{lemma}
The identity functor from $d \acal_{dual}$
to $d \acal_i$ is the left adjoint of
a Quillen equivalence.
\[
\xymatrix{
\id
:
d \acal_{{dual}}
\ar@<+1ex>[r] &
d \acal_{i}
:
\id \ar@<+0.5ex>[l]}
\]
\end{lemma}

The object $S^0$ is clearly dualizable. Similarly if $V$ is a finite dimensional
vector space, then $S^0 \otimes V$ is dualizable. As a consequence, we have the following lemma.

\begin{lemma}\label{lem:LRquillen}
There is a strong symmetric monoidal Quillen pair
\[
\xymatrix{
L
:
\ch(\qq)
\ar@<+1ex>[r] &
d \acal_{dual}
:
R \ar@<+0.5ex>[l]}
\]
where $L V = S^0 \otimes V$ and
$R A = \acal(S^0, A)_*$. Thus, $d \acal_{dual}$ is a closed
$\ch(\qq)$--model category.
\end{lemma}

\section{Obtaining an algebraic category}\label{sec:cyclic}

The method of this section is the synthesis of three ideas.
The first idea is
to use the Hasse-Tate square from the introduction to separate the homotopical information of
$\torus$--equivariant spectra into pieces where
we can remove equivariance without losing any information.

For $\torus$--equivariant spectra, the relevant decomposition is to
separate the homotopical information coming from finite subgroups
from the homotopical information coming from the whole group.
For this separation
we will need a diagram of model categories
rather than a diagram of commutative rings.
We establish the categorical foundations in
the next subsection and then perform the separation in
Subsection \ref{subsec:separate}.

The second is that the correct way to remove equivariance is to take fixed points.
The primary example is that taking $\torus$--fixed points
gives a Quillen equivalence from
$D E \torus_+$--modules in rational $\torus$--equivariant spectra to
$D \B \torus_+$--modules in rational spectra.
Here $D E \torus_+$ is the Spanier--Whitehead dual
of $E \torus_+$ in $\torus$--spectra
and $D \B \torus_+$ is the
Spanier--Whitehead dual of $\B \torus_+$ in the
category of spectra.
See Subsection \ref{subsec:fixedpoints}.

With the separation complete and equivariance removed,
we use the results of
\cite{shiHZ} to move to an algebraic setting in
Subsection \ref{subsec:toalgebra}.
That is, we obtain a Quillen equivalence between rational $\torus$--spectra
and some combined cellularization--localization of an algebraic category.

The next step is to simplify that algebraic category into the
algebraic model $\cA (\torus)$,
by directly calculating the effects of these cellularizations and localizations.
This is the essence of the third idea: to leave any examination of localizations
or cellularizations until one is working with an algebraic category.
This occurs in Subsection \ref{subsec:ramod}, where we simplify the category created by the results of \cite{shiHZ} and remove a localization.
Finally in Subsection \ref{subsec:torsion}
we remove a cellularization to get to the algebraic model.

\subsection{Diagrams of model categories}

We will use several model categories that are
built from diagrams of model categories.
This idea has been studied in some detail in  \cite{gsmodules}. In this section we introduce the
relevant structures and leave most of the proofs to the reference.
We will only use one shape of diagram, the pullback diagram $\pscr$:
\[
\bullet \longrightarrow \bullet \longleftarrow \bullet.
\]
Pullbacks of model categories are also considered in detail in
\cite{berg11}.

\begin{definition}
A \textbf{$\pscr$--diagram of model categories} $R^\bullet $ 
is a pair of Quillen pairs
\[
\begin{array}{rcl}
L:
\acal
&
\adjunct
&
\bcal : R \\
F:
\ccal
&
\adjunct
&
\bcal : G \\
\end{array}
\]
with $L$ and $F$ the left adjoints.
We will usually draw this as the diagram below.
\[
\xymatrix{
\acal
\ar@<+1ex>[r]^L
&
\bcal
\ar@<+0.5ex>[l]^R
\ar@<-0.5ex>[r]_G
&
\ccal
\ar@<-1ex>[l]_F
}
\]
\end{definition}

A standard example comes from a $\pscr$--diagram of rings
$R=(R_1 \overset{f}{\rightarrow} R_2 \overset{g}{\leftarrow} R_3)$. Using the adjoint pairs
of extension and restriction of scalars we obtain a $\pscr$-diagram
of model categories $R^\bullet $ as below.
\[
\xymatrix@C+0.3cm{
R_1 \leftmod
\ar@<+1ex>[r]^{R_2 \otimes_{R_1} -}
&
R_2 \leftmod
\ar@<+0.5ex>[l]^{f^*}
\ar@<-0.5ex>[r]_{g^*}
&
R_3 \leftmod.
\ar@<-1ex>[l]_{R_2 \otimes_{R_3} -}
}
\]

\begin{definition}
Given a $\pscr$--diagram of model categories $R^\bullet $
we can define a new category, $R^\bullet \leftmod$.
The objects of this category are pairs of morphisms,
$\alpha \co La \to b$ and $\gamma \co Fc \to b$ in $\bcal$.
We usually abbreviate a pair $(\alpha \co La \to b, \gamma \co Fc \to b)$ 
to a quintuple $(a,\alpha, b, \gamma, c)$. We find this notation
suggestive but  emphasize that objects of
$R^{\bullet}\leftmod$ are not usually modules over a diagram of
rings.

A morphism in $R^\bullet \leftmod$ from
$(a,\alpha, b, \gamma, c)$ to $(a',\alpha', b', \gamma', c')$
is a triple of maps
$x \co a \to a'$ in $\acal$,
$y \co b \to b'$ in $\bcal$,
$z \co c \to c'$ in $\ccal$ such that we have a commuting diagram in $\bcal$
\[
\xymatrix{
La \ar[r]^\alpha
\ar[d]^{Lx}
& b
\ar[d]^y
& Fc \ar[l]_\gamma
\ar[d]^{Fz} \\
La' \ar[r]^{\alpha'}
& b'
& Fc' \ar[l]_{\gamma'}
}
\]
\end{definition}

Note that we could also have defined an object as a sequence
$(a,\bar{\alpha}, b, \bar{\gamma}, c)$.
where $\bar{\alpha} \co a \to Rb$ is a map in $\acal$
and $\bar{\gamma} \co c \to Gb$ is a map in $\ccal$.

We say that a map $(x,y,z)$ in $R^\bullet \leftmod$ is an objectwise
cofibration if $x$ is a cofibration of $\acal$,
$y$ is a cofibration of $\bcal$ and
$z$ is a cofibration of $\ccal$.
We define objectwise weak equivalences similarly.

\begin{lemma}\cite[Proposition 3.3]{gsmodules}
Consider a $\pscr$--diagram of model categories $R^\bullet $ as below,
with each category cellular and proper.
\[
\xymatrix{
\acal
\ar@<+1ex>[r]^L
&
\bcal
\ar@<+0.5ex>[l]^R
\ar@<-0.5ex>[r]_G
&
\ccal
\ar@<-1ex>[l]_F
}
\]
The category of $R^\bullet$--modules admits a cellular
proper model structure with cofibrations and weak 
equivalences defined objectwise. This is called the {\em diagram injective} model structure. 
\end{lemma}

While there is also a diagram projective model structure, 
in this paper we only use the diagram injective model structure (and cellularizations thereof)
on diagrams of model categories.

Now consider maps of $\pscr$--diagrams of model categories.
Let $R^\bullet $ and $S^\bullet $ be two such diagrams,
where $R^\bullet $ is as above and $S^\bullet $ is given below.
\[
\xymatrix{
\acal'
\ar@<+1ex>[r]^{L'}
&
\bcal'
\ar@<+0.5ex>[l]^{R'}
\ar@<-0.5ex>[r]_{G'}
&
\ccal'
\ar@<-1ex>[l]_{F'}
}
\]
Now we assume that we have Quillen adjunctions as below
such that $P_2 L$ is naturally isomorphic to
$L' P_1$ and $P_2 F$ is naturally isomorphic to
$F' P_3$.
\[
\begin{array}{rcl}
P_1:
\acal
&
\adjunct
&
\acal' : Q_1 \\
P_2:
\bcal
&
\adjunct
&
\bcal' : Q_2 \\
P_3:
\ccal
&
\adjunct
&
\ccal' : Q_3 \\
\end{array}
\]
We then obtain a Quillen adjunction $(P,Q)$ between
$R^\bullet \leftmod$ and $S^\bullet \leftmod$.
For example, the left adjoint $P$ takes the object
$(a,\alpha,b,\gamma,c)$ to
$(P_1 a, P_2 \alpha, P_2 b, P_2 \gamma, P_3 c)$.
The commutativity assumptions ensure that this is an
object of $S^\bullet \leftmod$. It is easy to see the following

\begin{lemma}\label{lemma:QEondiag}
If the Quillen adjunctions $(P_i, Q_i)$ are
Quillen equivalences then the adjunction $(P,Q)$ between
$R^\bullet \leftmod$ and $S^\bullet \leftmod$ is a
Quillen equivalence.
\end{lemma}

Now we turn to monoidal considerations.
There is an obvious monoidal product for $R^\bullet \leftmod$,
provided that each of $\acal$, $\bcal$ and $\ccal$ is monoidal
and that the left adjoints $L$ and $F$ are strong monoidal.
\[
(a,\alpha, b, \gamma, c) \smashprod (a',\alpha', b', \gamma', c')
:=(a \smashprod a', \alpha \smashprod \alpha',
b \smashprod b',\gamma \smashprod \gamma',c \smashprod c')
\]
Let $S_\acal$ be the unit of $\acal$,
$S_\bcal$ be the unit of $\bcal$ and let
$S_\ccal$ be the unit of $\ccal$.
Since $L$ and $F$ are monoidal, we have maps
$\eta_\acal \co L S_\acal \to S_\bcal$
and $\eta_\ccal \co F S_\ccal \to S_\bcal$.
The unit of the monoidal product on
$R^\bullet \leftmod$ is
$(S_\acal, \eta_\acal, S_\bcal, \eta_\ccal, S_\ccal)$.

It is worth noting that this category has an internal function object
when $\acal$, $\bcal$ and $\ccal$ are closed monoidal categories and thus itself is closed.

\begin{lemma}
Consider a $\pscr$--diagram of model categories $R^\bullet $ such that each
vertex is a cellular monoidal model category.
Assume further that the two adjunctions of the diagram are strong monoidal
Quillen pairs.
Then $R^\bullet \leftmod$ is also a monoidal model category.
If each vertex also satisfies the monoid axiom,
so does $R^\bullet \leftmod$.
\end{lemma}
\begin{proof}
Since the cofibrations and weak equivalences are defined objectwise,
the pushout product and monoid axioms hold provided they do so
in each model category in the diagram $R^\bullet $.
\end{proof}

We can also extend our monoidal considerations to maps
of diagrams.
Return to the setting of a map $(P,Q)$
of $\pscr$--diagrams from $R^\bullet $ to $S^\bullet $ as described above.
If we assume that each of the adjunctions
$(P_1, Q_1)$, $(P_2, Q_2)$ and $(P_3, Q_3)$
is a symmetric monoidal Quillen equivalence, then we see that
$(P,Q)$ is a symmetric monoidal Quillen equivalence.

With these formalities out of the way, we are ready to
move from the model category of rational $\torus$ spectra to
modules over a $\pscr$--diagram of model categories.

\subsection{Isotropy separation}\label{subsec:separate}

In this subsection we separate the homotopical information of rational $\torus$--spectra
into three parts. The first part takes care of the homotopical information coming from
the finite cyclic subgroups. The second
part deals with the homotopical information coming from $\torus$.
The third part is a comparison term which enforces some compatibility
conditions on the two other parts.

We achieve this separation by replacing the category of
rational $\torus$--spectra with a
Quillen equivalent $\pscr$--diagram of model categories.

Before we do that, let us first recall some basic definitions and properties for $\torus$-spectra.

\begin{definition}
Let $\TSP$ be the category of $\torus$--equivariant orthogonal spectra
indexed on a complete $\torus$--universe $\ucal$ considered with the stable model structure. We denote it by $\TSP$.
\end{definition}

This model category is monoidal, proper and cellular \cite{mm02}. 
The weak equivalences are those maps $f$ such that $\pi_*^H(f)$ is an isomorphism
for all closed subgroups $H$ of $\torus$.

Following \cite[Section 5]{barnessplitting} and using \cite[Theorem IV.6.3]{mm02}, we localize this model category at the rational sphere spectrum $S_\qq$. 
That is, we leave the underlying category unchanged 
and alter the model structure. We call the weak equivalences of the localized model structure
\textbf{rational equivalences}:  a map $f$ is a rational equivalence 
if $\pi_*^H(f) \otimes \qq$ is an isomorphism for all closed subgroups $H$ of $\torus$. 
We call this model structure the \textbf{rational model structure} and use the notation $L_{S_\qq} \TSP$.

The localized model category is
still proper, cellular, monoidal and stable.

\begin{definition}\label{def:universalspaces}
Let $\fcal$ be the collection of finite cyclic subgroups of $\torus$.
There is a universal space for this family called $E \fcal$
where $E \fcal^H$ is non--equivariantly contractible
for each finite cyclic subgroup $H$
and $E \fcal^{K} = \emptyset$ for all other subgroups $K$.
We define $\widetilde{E} \fcal$ via the cofibre sequence of
$\torus$--spaces,
\[
E \fcal_+ \lra S^0 \lra \widetilde{E} \fcal.
\]
We define $DE \fcal_+$ to be $F(E \fcal_+, N^\#S)$. Here $N^\#$ is the lax monoidal right adjoint described in \cite[Theorem IV.3.9]{mm02} from EKMM $\torus$--equivariant $S$-modules to $\TSP$.
\end{definition}

Recall that $N^\#$ is the right adjoint of a Quillen equivalence
when $\TSP$ is considered with the positive stable model structure (see \cite[Chapter IV]{mm02} for more details). The spectrum $DE \fcal_+$ is
a commutative ring spectrum, which is fibrant in the positive stable model structure
on $\TSP$.

We can express the above cofibre sequence
as the Hasse-Tate homotopy pullback square
of $\torus$--equivariant spectra \cite[Section 17]{gremay95}.
To see that it is a homotopy pullback
square, note that the homotopy fibres of the top and
bottom row are weakly equivalent (where the bottom row is the top one smashed with $DE \fcal_+$)
\[
\xymatrix{
\sphspec \ar[r] \ar[d] &
\widetilde{E} \fcal \ar[d] \\
DE \fcal_+      \ar[r] &
DE \fcal_+ \smashprod \widetilde{E} \fcal .
}
\]

We have three model categories:
\begin{itemize}
\item $L_{S_\bQ}(DE \fcal_+ \leftmod)$, which captures the behaviour of the finite cyclic groups;
\item $L_{S_\bQ \smashprod \widetilde{E} \fcal} \TSP$,
which captures the behaviour of $\torus$;
\item $L_{S_\bQ \smashprod DE \fcal_+ \smashprod \widetilde{E} \fcal}
     (DE \fcal_+ \leftmod)$ which captures the interaction of the first two.
\end{itemize}
Now we can give our diagram of model categories that
separates the behaviour of the finite cyclic groups from the rest.

\begin{definition}
We define $S^\bullet$ to be the following diagram of model categories.
\[
\xymatrix@C+0.5cm{
L_{S_\bQ}(DE \fcal_+ \leftmod)
\ar@<+1ex>[r]^-{\Id}
&
L_{S_\bQ \smashprod DE \fcal_+ \smashprod \widetilde{E} \fcal}(DE \fcal_+ \leftmod)
\ar@<+0.5ex>[l]^-{\Id}
\ar@<-0.5ex>[r]_-{U}
&
L_{S_\bQ \smashprod \widetilde{E} \fcal} \TSP
\ar@<-1ex>[l]_-{DE \fcal_+ \smashprod -}
}
\]
Since all of the model categories in the diagram are
cellular, proper, monoidal model categories,
we have a cellular proper stable monoidal
model category $S^\bullet \leftmod$ that satisfies
the monoid axiom.
\end{definition}

Given an $X \in \TSP$, we have an $S^\bullet$--module
\[
S^\bullet \smashprod X
: =
(DE \fcal_+ \smashprod X, \id, DE \fcal_+ \smashprod X, \id , X).
\]
This functor has a right adjoint.
Let $A=(a,\alpha,b,\gamma, c)$ be an $S^\bullet$--module.
Then there are maps in $\TSP$: $ a \to b$ and $c \to DE \fcal_+ \smashprod c \to b $, where in the composite the first map is the unit of the adjunction $(DE \fcal_+ \smashprod - , U)$ and the second map is $\gamma$.
Thus we have a diagram in $\TSP$:
$a \rightarrow b \leftarrow c$. We write $\pullback A$
for the pullback of
this diagram in $\TSP$.
We assemble this construction into the following
result, the proof of which is entirely routine.

\begin{proposition}\label{prop:Smodulesadjunction}
There is a strong symmetric monoidal Quillen adjunction
\[
\xymatrix{
S^\bullet \smashprod - :
L_{S_{\bQ}}(\TSP)
\ar@<+1ex>[r] &
S^\bullet \leftmod
: \pullback \ar@<+0.5ex>[l]}
\]
\end{proposition}

We want to turn this adjunction into a Quillen equivalence.
To do so, we apply
the Cellularization Principle of  \cite[Proposition 2.7]{gscell}.
The idea is to cellularize (right Bousfield localize, see also Section \ref{subsec:cellularization})
the right hand model category so that this adjunction induces
a Quillen equivalence.

The generators for the homotopy category of
$L_{S_\bQ}(\TSP)$
are all suspensions and desuspensions of objects of the form
$\Sigma^\infty \torus/H_+$ for $H$ a subgroup of $\torus$.
For later purposes (see Section \ref{subsec:torsion}),
we want a set of cells with simpler algebraic models.
For every natural $n>1$ let
\[
\sigma_n=\torus \wedge_{C_n}e_{C_n}S^0
\]
where $e_{C_n}$ denotes the idempotent in the Burnside ring
for $C_n$ (cyclic group of order $n$)  corresponding to $C_n$.
By \cite[Lemma 2.1.5]{gre99},
\[
\torus/C_{n+}= \bigvee_{C_m \subseteq C_n}\sigma_m
\]
hence we know that the set
\[
K= \{ \Sigma^k S^0 \mid k \in \zz \} \bigcup \{\Sigma^k \sigma_n \mid \ n>1, \ k \in \zz \}
\]
is a set of (cofibrant and homotopically small) generators for
$L_{S_\bQ}(\TSP)$.

Let $K_{\Top}$ be the set of images
of the objects from $K$ under the functor $S^\bullet \smashprod-$ (up to isomorphism).
The elements of this set $K_{\Top}$ will be called basic \textbf{cells}.

To apply the Cellularization Principle we need to know that these cells are homotopically small
(this is also known as small or compact, see
Definition \ref{def:hocompact}).
First note that if $X$ is homotopically small in
$\TSP$ then it is so in $L_{S_\bQ} \TSP$
(since rationalization is a smashing localization).

Now consider the following three elements of $S^\bullet \leftmod$
\[
(\ast, \ast, DE\fcal_+ \smashprod X, \ast , \ast) \quad
(\ast, \ast, DE\fcal_+ \smashprod X, \id , X) \quad
(DE \fcal_+ \smashprod X, \id, DE\fcal_+ \smashprod X, \ast, \ast).
\]
It is routine to check that these are cofibrant and
homotopically small whenever $X$ is cofibrant
and homotopically small in $\TSP$.
Finally let $X$ be cofibrant in $\TSP$. There is a homotopy pushout diagram
as below, where the final term is $S^\bullet \smashprod X$.
\[
\xymatrix{
(\ast, \ast, DE \fcal_+ \smashprod X, \ast , \ast)
\ar[r] \ar[d] &
(\ast, \ast, DE \fcal_+ \smashprod X, \id , X)
\ar[d] \\
(DE \fcal_+ \smashprod X, \id, DE \fcal_+ \smashprod X, \ast, \ast)
\ar[r] &
(DE \fcal_+ \smashprod X, \id, DE \fcal_+ \smashprod X, \id , X)
}
\]
Homotopically small objects are preserved by homotopy pushouts
(consider the associated cofibre sequence). Hence
$S^\bullet \smashprod X$ is homotopically small in $S^\bullet \leftmod$
whenever $X$ is cofibrant and homotopically small.
Since these two conditions hold for the generators of $\TSP$,
we see that every element of $K_{\Top}$ is homotopically small.

Note that the model category $L_{DE \fcal_+ \smashprod \widetilde{E} \fcal}DE \fcal_+ \leftmod$ is the same as the model category $L_{\Sigma^* f}DE \fcal_+ \leftmod$ where $f:  DE \fcal_+  \longrightarrow DE \fcal_+ \smashprod \widetilde{E} \fcal$ and $\Sigma^* f$ denotes the set of all (integer) suspensions and desuspensions of $f$. This is a similar result to \cite[Lemma 4.14]{brstable}, since
$\widetilde{E} \fcal$--localization (in $\TSP$)
is given by smashing with the map
$S^0 \lra \widetilde{E} \fcal$.

\begin{proposition}\label{prop:Smodulesequivalence}
There is a Quillen equivalence
\[
\xymatrix{
S^\bullet \smashprod - :
L_{S_{\bQ}}(\TSP)
\ar@<+1ex>[r] &
K_{\Top} \cell S^\bullet \leftmod
: \pullback \ar@<+0.5ex>[l]}
\]
\end{proposition}
\begin{proof}
This follows from the Cellularization Principle,
\cite[Proposition 2.7]{gscell}.
It suffices to show that the derived unit is a weak equivalence on the
set $K$ of generators for the left hand side, which are shifts of the objects
$\sigma_n$ for $n>1$ and $S^0$.
Each such object is cofibrant and homotopically small,
as are the elements of $K_{\Top}$.

The derived left adjoint on cofibrant objects (such as the elements of $K$), is simply  the left adjoint. The right derived functor on objects
of the form $S^\bullet \smashprod k$ for $k \in K$ is weakly equivalent to taking a homotopy pullback of the following diagram:
\[
\xymatrix@R=2.5pc @C=4.5pc{
 &  S_\bQ \wedge \widetilde{E}\cF \wedge k
\ar[d]^{\id \wedge \id \wedge \lambda \wedge \mathrm{Id}} \\
S_\bQ \wedge DE\cF_+\wedge k
\ar[r]_(.45){\mathrm{Id} \wedge a \wedge \mathrm{Id} \wedge \mathrm{Id}} &
 S_\bQ \wedge \widetilde{E}\cF \wedge DE\cF_+\wedge k
}
\]
where the map $a \co S^0 \lra \widetilde{E}\cF$
is the map to the cofibre
and $\lambda$ is the unit map.
Since homotopy pullbacks commute with smash products, the homotopy pullback of the above is weakly equivalent to the homotopy pullback of
\[
\xymatrix@R=2pc{
 & \widetilde{E}\cF
\ar[d] \\
DE\cF_+ \ar[r] & DE\cF_+\wedge \widetilde{E}\cF
}
\]
(in the category $\TSP$) smashed with $S_\bQ \wedge k$.
But the homotopy pullback of the diagram above is $\sphspec$, as discussed above. Hence the derived unit is a weak equivalence (in $L_{S_\bQ}(\TSP)$) on the cells $k \in K$.
\end{proof}

We will show in Proposition \ref{prop:monoidalleft} below that this 
Quillen equivalence is actually a symmetric monoidal Quillen equivalence.

Thus we have separated the homotopical information of $\TSP$ into a diagram of three
model categories. The advantage of doing so is that we may now remove the equivariance
from the model category whilst keeping the correct homotopy category.

\subsection{Removing equivariance}\label{subsec:fixedpoints}

Now we are going to remove equivariance using the
inflation--fixed points adjunction.

Recall the functor $(-)^{\torus}$ of
\cite[Section 3]{mm02}. It takes
a spectrum indexed on a complete $\torus$--universe $\ucal$ to
the $\torus$--trivial universe $\ucal^{\torus}$ and then applies
the space--level fixed point functor levelwise.
We begin by extending this functor to categories of modules
over $\torus$-equivariant ring spectra.

If $A$ is a commutative ring spectrum in $\torus$--equivariant spectra then $A^{\torus}$ is a commutative ring object in spectra.
We want to compare $A$-modules in $\torus$--equivariant spectra and
$A^\torus$-modules in spectra.
Using \cite[Section 4]{gsfixed} there is a Quillen adjunction
\[
\xymatrix{
A \smashprod_{\varepsilon^* A^{\torus}} \varepsilon^* (-):
A^{\torus} \leftmod
\ar@<+1ex>[r] &
A \leftmod
: (-)^{\torus}. \ar@<+0.5ex>[l]}
\]
between right transferred model structures (fibrations
and weak equivalences are defined in terms of the underlying categories).
To simplify the notation, if $a \co \varepsilon^* A^{\torus} \to A$ is the 
inclusion of fixed points, we  write 
$$a_{\#} = A \smashprod_{\varepsilon^* A^{\torus}} \varepsilon^* (-)$$
for  the left adjoint.

We consider several cases of this kind of adjunction and use them
to build up an adjunction between $S^\bullet \leftmod$
and a new diagram of model categories  $S^\bullet_{\Top} \leftmod$.
We then show that this adjunction gives a Quillen equivalence,
after cellularising.

\begin{proposition} For $a \co \varepsilon^* DE\cF_+^{\torus} \to DE\cF_+$ the inclusion of fixed points, 
the adjunction
\[
\xymatrix@C=4pc{
{a_{\#}}: L_{S_\bQ}(DE\cF_+^{\torus} -\mathrm{mod})\  \ar@<+1ex>[r] & \ L_{S_\bQ}(DE\cF_+ -\mathrm{mod}): (-)^{\torus} \ar@<+0.5ex>[l]
}
\]
is a symmetric monoidal Quillen equivalence.
\end{proposition}
\begin{proof}
We have a Quillen equivalence by \cite[Corollary 8.1 and Corollary 9.2]{gsfixed}. The left adjoint is strong symmetric monoidal
so the result follows.
\end{proof}

We now left Bousfield localize the model categories in this adjunction. We localize the right hand side at the set of maps
$\Sigma^* f$, where $f: DE\cF_+ \lra DE\cF_+\wedge \widetilde{E}\cF$.
Let $(\Sigma^* f)^{\torus}$ be the set of maps obtained by applying the derived right adjoint to the maps in $\Sigma^* f$.
By \cite[Theorem 3.3.20, part 1b]{hir03} we
obtain the following result.
\begin{proposition}The adjunction
\[
\xymatrix@C=4pc{
{a_{\#}}: L_{(\Sigma^* f)^{\torus}}L_{S_\bQ}(DE\cF_+^{\torus} -\mathrm{mod})\  \ar@<+1ex>[r] & \ L_{\Sigma^* f}L_{S_\bQ}(DE\cF_+ -\mathrm{mod}): (-)^{\torus} \ar@<+0.5ex>[l]
}
\]
is a symmetric monoidal Quillen equivalence.
\end{proposition}

Our final version is where we take $A$ to be the sphere spectrum,
so the left adjoint is just $\varepsilon^*$.
By \cite[Section V, Proposition 3.10]{mm02} the adjunction
\[
\xymatrix@C=4pc{
{\varepsilon^*}: {\Sp^\ocal}\  \ar@<+1ex>[r] & \ \torus \Sp^\ocal: (-)^{\torus} \ar@<+0.5ex>[l]
}
\]
is a symmetric monoidal Quillen adjunction. We localize it to obtain a Quillen equivalence

\begin{proposition}The adjunction
\[
\xymatrix@C=4pc{
{\varepsilon^*}: L_{S_\bQ}(\Sp^\ocal)\  \ar@<+1ex>[r] & \ L_{S_\bQ \wedge \widetilde{E}\cF}(\TSP): (-)^{\torus} \ar@<+0.5ex>[l]
}
\]
is a symmetric monoidal Quillen equivalence.
\end{proposition}
\begin{proof}Since $\varepsilon^*$ is strong monoidal and $\varepsilon^*(S_\bQ)=S_\bQ$ the above adjuncton is a composite of two adjunctions, the second being identity adjunction between $L_{S_\bQ}(\TSP)$ and further localization at $\widetilde{E}\cF$, namely $L_{S_\bQ \wedge \widetilde{E}\cF}(\TSP)$.

To verify that this is a Quillen equivalence we will work with the derived unit and the derived counit on generators. The generator for the left hand side is $S^0$.
The generators for the right hand side are
$S^0=\torus /\torus _+$  and $(\torus /C_{n})_+$ for $n \geqslant 1$.
But $(\torus /C_{n})_+$ is weakly equivalent to a point in
$L_{S_\bQ \wedge \widetilde{E}\cF}(\TSP)$
(that is, $(\torus /C_{n})_+ \smashprod \widetilde{E}\cF \simeq \ast$).
So we only need to consider $S^0$ for the right hand side.

The derived functor of $(-)^\torus$ acts as the geometric
$\torus$-fixed point functor, because
$\phi^N(X)=(X\wedge \widetilde{E}[\not \supseteq N])^N$.
With this in mind, it is routine to check that
the derived unit and counit are weak equivalences
on the generators.
It follows that this adjunction is a Quillen equivalence.
\end{proof}

We can extend the functor $(-)^{\torus}$ to the level of generalized diagrams.

\begin{definition}
We define $S^\bullet_{\Top} $ to be the following diagram of model categories and adjoint Quillen pairs:
\[
\xymatrix@C=4pc{
 L_{S_\bQ}(DE\cF_+^{\torus} -\mathrm{mod})\
\ar@<+1ex>[r]^-{\mathrm{Id}}
&
\ L_{\{(\Sigma^* f)^{\torus}\}}L_{S_\bQ}(DE\cF_+^{\torus} -\mathrm{mod})
\ar@<+0.5ex>[l]^-{\mathrm{Id}}
\ar@<-1ex>[r]_-{U}
&
\ L_{S_\bQ}(\TSP)
\ar@<-0.5ex>[l]_-{DE\cF_+^{\torus} \wedge-}
}
\]
where $U$ denotes the forgetful functor.
\end{definition}

Now we consider the category $S^\bullet_{\Top} \leftmod$ of generalized diagrams. By construction,
it follows that the functor $(-)^\torus$
induces a functor between $S^\bullet \leftmod$ and
$S^\bullet_{\Top} \leftmod$. For simplicity, we call the left adjoint $a_{\#}$.
Since each of the components is
a symmetric monoidal Quillen equivalence, we obtain the following from Lemma \ref{lemma:QEondiag}.

\begin{theorem}\label{thm:removeequivariance}The adjunction
\[
\xymatrix@C=4pc{
{a_{\#}}: S^\bullet_{\Top} \leftmod\  \ar@<+1ex>[r] & \ S^\bullet \leftmod: (-)^{\torus} \ar@<+0.5ex>[l]
}
\]
 is a symmetric monoidal Quillen equivalence.
\end{theorem}

We now extend this Quillen equivalence to a cellularized version.
Define $K_{\Top}^{\torus}$ to be the set of cells given by
applying the derived functor of $(-)^{\torus}$ to
$K_{\Top}$. By the Cellularization Principle of the
 \cite[Proposition 2.7]{gscell},
we see that the Quillen equivalence
above is preserved by cellularization.

\begin{corollary}\label{cor:removeequivariancecell}
The adjunction below is a Quillen equivalence.
\[\xymatrix@R+0.5cm@C=4pc{
{a_{\#}}\ :\ K_{\Top}^{\torus} \cell S^\bullet_{\Top} \leftmod
\ar@<+1ex>[r]
&
K_{\Top} \cell S^\bullet \leftmod \ :\ {(-)^{\torus}}
\ar@<+0.5ex>[l]
}\]
\end{corollary}

As in the previous section, the above Quillen equivalence is symmetric monoidal, but for clarity we postpone the proof of that fact to Section \ref{sec:application}.

The model category $K_{\Top}^{\torus} \cell S^\bullet_{\Top} \leftmod$ is constructed from model categories
of non-equivariant spectra. Hence we have
removed the equivariance. The reward for doing so is
in the next section, where we can now replace
our categories based on spectra
with categories based on rational chain complexes.
Such categories are our first approximation to
the algebraic model.

%%%%%%%%%%%%%%%%%%%%%%%%%%%%%%%%%%%%%%%%%%%%%%%%

\subsection{Passing to algebra}\label{subsec:toalgebra}

We will replace the model category
$K_{\Top}^{\torus} \cell S_{\Top}^\bullet \leftmod$
by a Quillen equivalent $\ch(\qq)$--model category.
The results of \cite{shiHZ}
and the general theory of diagrams of model categories
allow us to do so. To apply the work of \cite{shiHZ}, we must work with $\h \qq$--modules in symmetric spectra.
So we give two Quillen equivalences, the first moves us from orthogonal spectra
to symmetric spectra, the second from symmetric spectra
to $\h \qq$--modules.

In more detail, recall $\bU$, the forgetful functor from orthogonal spectra (in based topological spaces) to symmetric spectra (in based simplicial sets) and call $\bP$ its left adjoint.
Define $\bU S^\bullet_{\Top} $ to be the diagram of model categories
\[
\xymatrix@C=3pc @R=3pc{
 L_{S_\bQ}(\bU DE\cF_+^{\torus} -\mathrm{mod})\
\ar@<+1ex>[r]^-{\mathrm{Id}}
&
\ L_{\{ \bU (\Sigma^* f)^{\torus}\}}L_{S_\bQ}(\bU DE\cF_+^{\torus} -\mathrm{mod})
\ar@<+0.5ex>[l]^-{\mathrm{Id}}
\ar@<-1ex>[r]_-{U}
&
\ L_{S_\bQ}\mathrm{Sp}^{\Sigma}
\ar@<-0.5ex>[l]_-{\bU DE\cF_+^{\torus} \wedge-}
}.
\]
The functor $\bU$ preserves all weak equivalences, so we do not need to apply fibrant replacement when constructing the set $\bU (\Sigma^* f)^{\torus}$ and the
commutative ring spectrum $\bU DE\fcal_+^{\torus}$.

\begin{proposition}The adjunction
\[
\xymatrix{
{\bU}^{\bullet}: S^\bullet_{\Top}-\mathrm{mod}\  \ar@<-1ex>[r] & \ \bU S^\bullet_{\Top} -\mathrm{mod}: \bP^\bullet \ar@<-0.5ex>[l]
}
\]
is a strong symmetric monoidal Quillen equivalence.
\end{proposition}
\begin{proof}
The adjunction $(\bP, \bU)$  is a Quillen equivalence between $L_{S_\bQ}\Sp^\ocal$  and $L_{S_{\bQ}} \mathrm{Sp}^{\Sigma}$. Furthermore the left adjoint is strong symmetric monoidal, so the result follows by Lemma \ref{lemma:QEondiag}.
\end{proof}

The second step is to pass from symmetric spectra to $H\bQ$ modules using the adjunction $(\h \qq \smashprod -, U)$. This is a Quillen equivalence between $L_{S_\bQ}\mathrm{Sp}^\Sigma$ and $\h \qq \leftmod$ and the left adjoint is strong symmetric monoidal. Thus by the same argument as above we get the following.

  \begin{proposition}The adjunction
 \[
\xymatrix{
{H\bQ\wedge -}^{\bullet}: \bU S^\bullet_{\Top}-\mathrm{mod}\  \ar@<+1ex>[r] & \ H\bQ \wedge \bU S^\bullet_{\Top} -\mathrm{mod}: \mathrm{U^\bullet} \ar@<+0.5ex>[l]
}
\]
  is a strong symmetric monoidal Quillen equivalence, where $H\bQ \wedge \bU S^\bullet_{\Top} $ denotes the following diagram of model categories:

   \[
\xymatrix@C=2pc @R=3pc{
 H\bQ \wedge \bU DE\cF_+^{\torus} -\mathrm{mod}\
\ar@<+1ex>[r]^-{\mathrm{Id}}
&
\ L_{\{H\bQ \wedge \bU (\Sigma^* f)^{\torus}\}}(H\bQ \wedge \bU DE\cF_+^{\torus} -\mathrm{mod})
\ar@<+0.5ex>[l]^-{\mathrm{Id}}
\ar@<-1ex>[r]_-{U}
&
\ H\bQ \leftmod
\ar@<-0.5ex>[l]_-{\bU DE\cF_+^{\torus} \wedge-}
}
\]
Where all $H\bQ \wedge \bU DE\cF_+^{\torus}$ denotes first the cofibrant replacement in the model category of commutative ring spectra and then application of $H\bQ \wedge -$.
\end{proposition}

\begin{rmk}
It is essential for the formality argument in Section \ref{subsec:ramod} that the ring spectrum
$H\bQ \wedge \bU DE\cF_+^{\torus}$ is commutative. Without this, 
one is unable to replace the ring $S_t$ (defined below) by the simpler ring $\ocal_\fcal$,
nor can one understand the localising set $A''$ (defined in the next Section) in terms of the inclusion 
$\ocal_\fcal \to \ecal^{-1} \ocal_\fcal$.
\end{rmk}

Now we are ready to use the results from \cite{shiHZ} to move from topology to algebra on generalized diagrams.
Let $\Theta$ denote the derived functor described in \cite[Section 2.2]{shiHZ}. This functor $\Theta$ induces an equivalence
between $\h \qq$-modules and rational chain complexes.
By \cite[Theorem 1.2]{shiHZ}
there is a commutative rational differential graded algebra
$S_t$, which is naturally weakly equivalent to
$\Theta (H\bQ \wedge \bU DE\cF_+^{\torus})$ such that
the model category of $S_t$--modules (in $Ch(\bQ)$) is Quillen equivalent to
the model category of $H\bQ \wedge \bU DE\cF_+^{\torus}$-modules (in $\torus$-spectra).

Let $S^\bullet_t $ be the diagram of model categories below,
where $\Theta (H\bQ \wedge \bU (\Sigma^* f)^{\torus})$
denotes the image of the set of maps
$H\bQ \wedge \bU (\Sigma^* f)^{\torus}$
in the category of $S_t$--modules under the derived functor.
\[
\xymatrix@C=4pc @R=3pc{
S_t \leftmod
\ar@<+1ex>[r]^-{\mathrm{Id}}
&
L_{\Theta (H\bQ \wedge \bU (\Sigma^* f)^{\torus})}
(S_t \leftmod)
\ar@<+0.5ex>[l]^-{\mathrm{Id}}
\ar@<-1ex>[r]_-{U}
&
 \ch(\qq)
\ar@<-0.5ex>[l]_-{S_t \otimes -}
}
\]

\begin{proposition}\label{prop:cosik}
There is a zig-zag of symmetric monoidal Quillen equivalences 
$${H\bQ \wedge \bU S^\bullet_{\Top} \leftmod}\simeq S_t^\bullet \leftmod.$$
 \end{proposition}

\begin{proof}
There is a zig-zag of symmetric monoidal adjunctions
between $\h\qq$--modules and $\ch(\qq)$.
By \cite[Corollary 2.15]{shiHZ} this zig-zag consists of Quillen equivalences. We can extend this zig-zag from
$\h \qq$--modules to
$H\bQ \wedge \bU DE\cF_+^{\torus}$--modules
in a natural way.

We can extend further to diagrams of model categories.
Thus we obtain a zig-zag of adjunctions between
${H\bQ \wedge \bU S^\bullet_{\Top} \leftmod}$ and
$S_t^\bullet \leftmod$.
At each stage, we have localized the middle category of the diagram at
the derived image (i.e. image under the derived functor) of the set of maps
$\{H\bQ \wedge \bU (\Sigma^* f)^{\torus}\}$.
We apply Lemma \ref{lemma:QEondiag}
to see that we have a symmetric monoidal Quillen equivalence
as claimed.
\end{proof}

\begin{corollary}\label{cor:gettingtoalgebra}
Denote the derived images (i.e. images under the derived functor) of the cells
$K^{\torus}_{\Top}$ in
$S^\bullet_t \leftmod$ by $K_t$.
Then there is a zig-zag of
Quillen equivalences 
\[
K_{\Top}^{\torus} \cell S^\bullet_{\Top} \leftmod \simeq  K_t \cell S_t^\bullet \leftmod.
\]
\end{corollary}

Since cellularization is compatible with Quillen
equivalences , i.e. all Quillen equivalences presented above are still Quillen equivalences after cellularising at the derived images of the cells from the set $K^{\torus}_{\Top}$. By the discussion in Section \ref{subsec:cellularization} and \ref{sec:application} the above zig-zag consists of symmetric monoidal Quillen equivalences.

%%%%%%%%%%%%%%%%%%%%%%%%%%%%%%%%%%%%%%%%%%%%
\section{Simplifying the algebraic category}

We have shown so far that the category of rational $\torus$-spectra has an algebraic model of the form  $K_t \cell S_t^\bullet \leftmod$. However, since this category is not well understood, in this section we perform several steps to obtain a more concrete and easier algebraic model.

\subsection{Removing the localization}\label{subsec:ramod}

In this section we have two tasks: replace the
commutative dga $S_t$ by something simpler
and remove the localization of the middle
model category.

The main idea is to do a formality argument, similar to the one in
\cite[Section 10]{tnqcore}. However, the important difference lies in
adapting the formality argument to one for  {\em modules} over a commutative dga. This is enough to simplify the middle model category in $S_t^\bullet $. 

The construction of $\Theta$ comes with
an isomorphism between $\h_* (\Theta X)$ and
$\pi_* (X)$ for any $\h\qq$--module $X$.
It follows that the homology of $S_t$ is determined by the
rational homotopy groups of $DE \fcal_+^\torus$.
We prove that the homology
of $S_t=\theta(\h \qq \smashprod \bU DE\fcal^{\torus}_+)$ is so well--structured that $S_t$ is quasi--isomorphic to its homology.
We then use this to understand the set of maps
$A=\Theta (H\bQ \wedge \bU (\Sigma^* f)^{\torus})$.

Recall that  $\ocal_\fcal$ is the graded ring
$\prod_{n \geqslant 1} \qq[c_n]$
with each $c_n$ of degree $-2$ and
$\ecal^{-1}\ocal_\fcal$ is the colimit over $n$ of
$\ocal_\fcal[c_1^{-1}, \dots, c_n^{-1}]$,
see Section \ref{subsec:the model}.

\begin{lemma}
We have isomorphisms of graded rings
\[
\begin{array}{ccccccc}
\h_* (S_t)
& \cong &
\h_* (\Theta(\h \qq \smashprod \bU (DE\fcal_+^\torus)) )
& \cong &
\pi_* (\h \qq \smashprod \bU (DE\fcal_+^\torus))  \\
& \cong &
\pi_* (DE\fcal_+^\torus) \otimes \qq
& \cong &
\pi_*^{\torus} (DE\fcal_+)\otimes \qq
& \cong &
\ocal_\fcal.
\end{array}
\]
\end{lemma}
Note that for the step
$\pi_* (DE\fcal_+^\torus) \otimes \qq
\cong
\pi_*^{\torus} (DE\fcal_+)\otimes \qq $
we require $DE\fcal_+$ to be a (positive) fibrant
spectrum.

We want to create a zig--zag of quasi--isomorphisms
between $S_t$ and $\ocal_\fcal$.
For each $n \geqslant 1$ there is a cycle $x_H$
inside $S_t$ which represents $e_n$ (projection onto factor $n$)
in homology.
It follows that the homology of
$S_t [(x_n)^{-1}]$ is
equal to $e_n$ applied to the homology of $S_t$.
Note that for this argument to hold, we need to know that $S_t$
is a commutative dga, which requires that
$DE \fcal_+$ be a commutative
ring object in $\torus$--spectra.

Define $\widetilde{S}_t = \prod_{n \geqslant 1} S_t [x_n^{-1}]$.
There is a canonical map $\alpha \co S_t \to \widetilde{S}_t$
which is a homology isomorphism.
For each $n \geqslant 1$, pick a representative $a_n$
in $S_t [x_n^{-1}]$ for the homology class of $c_n$.
We thus have a map
$\qq[c_n] \to S_t [x_n^{-1}]$ which sends
$c_n$ to $a_n$.
Define $\beta \co \ocal_\fcal \to \widetilde{S}_t$
as the product over $n$ of these maps.
We now have our zig-zag of quasi-isomorphisms.

Let $A'$ be the image of the set $A$ under
(derived) extension of scalars along $\alpha$.
Define a new diagram of model categories,
$\widetilde{S}_t^\bullet $ as
\[
\xymatrix@C=4pc @R=3pc{
\widetilde{S}_t \leftmod
\ar@<+1ex>[r]^-{\mathrm{Id}}
&
L_{A'}
(\widetilde{S}_t \leftmod)
\ar@<+0.5ex>[l]^-{\mathrm{Id}}
\ar@<-1ex>[r]_-{U}
&
\ch(\qq)
\ar@<-0.5ex>[l]_-{S_t \otimes -}
}
\]
Extension and restriction of scalars along $\alpha \co S_t \to \widetilde{S}_t$
induce a symmetric monoidal
Quillen equivalence between $S_t^\bullet \leftmod$
and $\widetilde{S}_t^\bullet \leftmod$.

We repeat this construction once more using $\beta$.
Let $A''$ be the image of the set $A'$ under
restriction of scalars along $\beta$.
Define a new diagram of model categories,
$\widetilde{S}_a^\bullet $ as
\[
\xymatrix@C=4pc @R=3pc{
\ocal_\fcal \leftmod
\ar@<+1ex>[r]^-{\mathrm{Id}}
&
L_{A''}
(\ocal_\fcal\leftmod)
\ar@<+0.5ex>[l]^-{\mathrm{Id}}
\ar@<-1ex>[r]_-{U}
&
\ch(\qq)
\ar@<-0.5ex>[l]_-{S_t \otimes -}
}
\]
Extension and restriction of scalars along $\beta \co \ocal_\fcal \to \widetilde{S}_t$
induce a symmetric monoidal
Quillen equivalence between $\widetilde{S}_{\widetilde{a}}^\bullet \leftmod$
and $\widetilde{S}_t^\bullet \leftmod$.

We summarise these results in the following.

\begin{proposition}
The adjoint pairs of extension and restriction of scalars
along $\alpha$ and $\beta$ induce symmetric monoidal
Quillen equivalences 
$$S_t^\bullet \leftmod\simeq \widetilde{S}_t^\bullet \leftmod\simeq \widetilde{S}_{\widetilde{a}}^\bullet \leftmod.$$

Let $K_{\widetilde{t}}$ be the derived images of the cells $K_t$ in
$\widetilde{S}_t^\bullet \leftmod$
and $K_{\widetilde{a}}$ the derived images in
$\widetilde{S}_a^\bullet \leftmod$.
Then we have Quillen equivalences between
$K_{t} \cell S_t^\bullet \leftmod$,
$K_{\widetilde{t}} \cell U^\bullet \leftmod$ and
$K_{\widetilde{a}} \cell \widetilde{S}_a^\bullet \leftmod$.
\end{proposition}

%%%%%%%%%%%%%%%%%%%%%%%%%%%%%%%%%%%%%%%%%%%%%%%%

Our next task is to understand the set of maps in $A''$
so that we can remove the localization in the middle model category in the diagram of model categories
 $\widetilde{S}_a^\bullet $.
We show that there is a zig-zag of homology isomorphisms between
\[
{ \theta (\bU DE\cF_+^{\torus}) \lra \theta (\bU (\widetilde{E}\cF \wedge DE\cF_+)^{\torus})}
\quad \text{and} \quad
j \co \ocal_\fcal \lra \ecal^{-1} \ocal_\fcal.
\]
It will follow that we can replace the set $A''$
by the set of all shifts of $j$ without changing
the effect of the localization.
That is, we will show that the model categories
$L_{A''}(\cO_\cF \leftmod)$ and
$L_{\Sigma^* j} (\cO_\cF \leftmod)$ are equal.

The zig-zag of homology isomorphisms of $\cO_\cF$ modules that we will use is as follows.
Factor $\theta (\bU DE\cF_+^{\torus}) \lra \theta (\bU (\widetilde{E}\cF \wedge DE\cF_+)^{\torus})$
into cofibration followed by an acyclic fibration
(with intermediate term $R$).
Let $C$ be the pushout of the square below.
Since $\cO_\cF \leftmod$ is left proper
it follows that $R \lra C$ is a quasi-isomorphism.
\[
\xymatrix@C=4pc@R=1.2pc{
\theta (\bU DE\cF_+^{\torus})
\ar@{>->}[r]
\ar[d]_\simeq &
R
\ar[d]^{\simeq}
\ar[r]^-\simeq &
\theta (\bU (\widetilde{E}\cF \wedge DE\cF_+)^{\torus})\\
\widetilde{S}_t \ar[r]^{f} & C \ar@{}[ul]|(.2){\ulcorner} \\
 \cO_\cF
\ar[u]^{\simeq}
\ar[r]^a \ar[d]_{\Id} &
C \ar[u]_{\id} \ar[d] \\
 \cO_\cF \ar[r] & \cE^{-1}C \\
 \cO_\cF \ar[u]^{\id} \ar[r]^{j} &
\cE^{-1}\cO_\cF \ar[u]^{\ecal^{-1} a}
 }
\]
The functor defined by $M \mapsto \ecal^{-1}M $
on $\ocal_\fcal$--modules $M$ is exact.
It follows that $C \lra \ecal^{-1} C$ is a homology isomorphism,
since $\ecal^{-1}$ is already inverted on homology.
The map $f$ induces a homology isomorphism once $\ecal$ has
been inverted, hence so does $a$.
It follows that $\ecal^{-1} a$ is a  homology isomorphism.

Thus we have shown that model categories $L_{A''} (\ocal_\fcal\leftmod)$
and $L_{\Sigma^* j}(\cO_\cF \leftmod)$ are equal. Now we are ready to remove the localization all together.

\begin{lemma}\label{zamiastLocBedzieInnyPierscien}
The adjunction induced by the inclusion
of rings $j: \cO_\cF \lra \cE^{-1}\cO_\cF$
induces a symmetric monoidal Quillen equivalence
 \[
\xymatrix{
\cE^{-1}\cO_\cF \otimes_{\cO_\cF} -\ :\
L_{\Sigma^* j}(\cO_\cF \leftmod)
 \ar@<+1ex>[r] & \
\cE^{-1}\cO_\cF \leftmod\ :\ j^* \ar@<+0.5ex>[l]
}.
\]
\end{lemma}
\begin{proof}
The cofibrations are unchanged by localization.
The weak equivalences of the model category $L_{\Sigma^* j}(\cO_\cF \leftmod) $
are those maps $f$ such that
\[
H_* (\cE^{-1}\cO_\cF \otimes_{\cO_\cF} f) =
\ecal^{-1} H_* (f)
\]
is an isomorphism.
The left adjoint preserves (and detects) these new weak equivalences,
so we have a symmetric monoidal Quillen adjunction as claimed.
The object $\cE^{-1}\cO_\cF$ is a
homotopically small generator for
(the homotopy category of) $\cE^{-1}\cO_\cF \leftmod$.
If we can show that the derived counit of this
adjunction is a weak equivalence then it will follow
that we have a Quillen equivalence.
This follows since the
counit map is an isomorphism on the generator:
\[
\cE^{-1}\cO_\cF \otimes_{\cO_\cF} \cE^{-1}\cO_\cF
\lra
\cE^{-1}\cO_\cF.
\]
\end{proof}

We use the above result to remove the localization from the middle term in our diagram of model categories. We have a commuting diagram of model categories as below,
where $U$ denotes the forgetful functor.
\[
\xymatrix@C=4pc @R=3pc{
\cO_\cF \leftmod\
\ar@<+1ex>[r]^-{\id}
\ar@<-1ex>[d]_{\id}
&
L_{\Sigma^* j}(\cO_\cF \leftmod)
\ar@<+0.5ex>[l]^-{\id}
\ar@<-1ex>[r]_-{U}
\ar@<-1ex>[d]_{\cE^{-1}\cO_\cF \otimes_{\cO_\cF} -}
&
 \bQ \leftmod
\ar@<-0.5ex>[l]_-{\cO_\cF \otimes -}
\ar@<-1ex>[d]_{\id}\\
\cO_\cF \leftmod\
\ar@<+1ex>[r]^-{\cE^{-1}\cO_\cF \otimes_{\cO_\cF} -}
\ar@<-1ex>[u]_{\id}
&
\cE^{-1}\cO_\cF \leftmod
\ar@<+0.5ex>[l]^-{j^*}
\ar@<-1ex>[r]_-{U}
\ar@<-1ex>[u]_{j^*}
&
\bQ \leftmod
\ar@<-0.5ex>[l]_-{\cE^{-1}\cO_\cF \otimes - }
\ar@<-1ex>[u]_{\id}
}
\]
We denote the bottom row by $S_a^\bullet $, the left adjoint from top to bottom by $\cE^{-1}\cO_\cF \otimes_{\cO_\cF} -$, the right adjoint by $j^*$ and we summarise the above in the following

\begin{proposition}The adjunction (described above)
 \[
\xymatrix{
\cE^{-1}\cO_\cF \otimes_{\cO_\cF} -\ :\
\widetilde{S}^\bullet_a \leftmod \
\ar@<+1ex>[r] &
\ S^\bullet_a \leftmod \ :\ j^* \ar@<+0.5ex>[l]
}
\]
is a symmetric monoidal Quillen equivalence, and thus the adjunction
 \[
\xymatrix{
\cE^{-1}\cO_\cF \otimes_{\cO_\cF} -\ :\
K_{\widetilde{a}} \cell \widetilde{S}^\bullet_a\leftmod \  \ar@<+1ex>[r] & \
K_a  \cell S^\bullet_a\leftmod \ :\ j^* \ar@<+0.5ex>[l]
}
\]
is a Quillen equivalence, where $K_a$ is the derived image of $K_{\widetilde{a}}$ under the left adjoint.
\end{proposition}

Again the adjunction at the level of cellularized categories is a symmetric monoidal Quillen equivalence by discussion in Section \ref{sec:application}.

 %%%%%%%%%%%%%%%%%%%%%%%%%%%%%%%%%%%%%%%%%%%%%%%%

\subsection{Removing the cellularization}\label{subsec:torsion}

We now compare $K_a  \cell S_a^\bullet \leftmod$
and the algebraic model $d \acal_{dual}$
of Section \ref{sec:cyclicmodel}. The point is to move
from a category whose weak equivalences are
quite complicated to define to a model category
whose weak equivalences are the quasi-isomorphisms. The idea behind this step is similar to one presented in Section 12 and 13 of \cite{tnqcore}.

We first introduce an adjoint pair relating
$S_a^\bullet \leftmod$ and
$d \hat{\acal}$.
An object
\[
\beta \co M \lra \ecal^{-1} \ocal_\fcal \otimes V
\]
of $d \hat{\acal}$ gives
an object of $S_a^\bullet \leftmod$
defined by
\[
(M, \ecal^{-1}\beta, \ecal^{-1} \ocal_\fcal \otimes V, \id , V)
\]
This functor, which we call $k^*$, includes $d \hat{\acal}$ into
$S_a^\bullet \leftmod$.
It has a right adjoint $\Gamma_v$.
Let $(a, \alpha, b, \gamma, c)$ be an object of
$S_a^\bullet \leftmod$. Then we can draw the diagram
of $\ocal_\fcal$--modules
\[
a \lra \ecal^{-1} \ocal_\fcal \otimes_{\ocal_\fcal} a
\lra b \longleftarrow \ecal^{-1} \ocal_\fcal \otimes c.\]
If we take the pullback  $P$ of this in the category of
$\ocal_\fcal$--modules in $\ch(\qq)$ we obtain a map
$\delta \co P \to \ecal^{-1} \ocal_\fcal \otimes c$.
This map $\delta$ is an object of $d \hat{\acal}$.
For more details see
\cite[Section 7]{gre12standardalgebra}.
We call this adjoint pair $(k^*, \Gamma_v)$ and we note that it is a
strong symmetric monoidal adjunction.

We can compose this adjunction with the adjunction
$(j^*, \Gamma_h)$ which relates $d \hat{\acal}$
to $d \acal$ (see Section \ref{subsec:adjunctions}).
We let $l^*= k^* \circ j^*$ and $\Gamma = \Gamma_h \circ \Gamma_v$.

\begin{lemma}
There is a symmetric monoidal adjunction $(l^*, \Gamma)$ between
$d \acal$ and $S_a^\bullet \leftmod$.
\end{lemma}
This adjunction is also studied in \cite[Section 7]{barnesmonoidal}
where it is called $(\textrm{inc}, \Gamma)$ and $S_a^\bullet $ is called
$R_a^\bullet$.

Recall, that up to a weak equivalence (and ignoring shifts), the cells $K_{\Top}$
consist of objects of the form
\[
S^\bullet \smashprod k =
\big(
k \smashprod DE \fcal_+
\lra
k \smashprod DE \fcal_+ \smashprod \widetilde{E} \fcal
\longleftarrow
k \smashprod \widetilde{E} \fcal
\big) .
\]
where $k \in K$, i.e. $k=S^0$ or $k=\sigma_n$ for $n>1$ (see Section \ref{subsec:separate}).

Thus we have to calculate the cells in $K_a$, i.e. the derived images  of cells from $K$ (or equivalently from $K_{\Top}$) in $S_a^\bullet \leftmod$. Since all required Quillen equivalences are symmetric monoidal (which follows from Section \ref{sec:monoidal}), they preserve the unit (up to weak equivalence), and the unit is always cellular.  So the derived image of $S^0$ is the unit in $S_a^\bullet \leftmod$:
\[
\ocal_\fcal \lra \ecal^{-1} \ocal_\fcal \longleftarrow \qq.
\]
We will use a simplified notation $S^0$ for it.
As for the other cells, consider some
$S^\bullet \smashprod \sigma_n \in K_{\Top}$.
One takes homotopy $\torus$--fixed points to get an object of
$K_{\Top}^\torus$ and then one applies all the derived functors of the fourth author
\cite{shiHZ}, to get an object of $K_{t}$. Finally one applies
a number of algebraic adjunctions to get an object of $K_a$.
Let the image of $S^\bullet \smashprod \sigma_n$ under these derived adjunctions
be $k_n = (A \to B \leftarrow C)$.
All of these adjunctions are constructed by taking Quillen equivalences
(which preserve the unit up to weak equivalence) on each of the component categories.
It follows that we have isomorphisms as below.
\[
\h_*(A)=
[\ocal_\fcal, A]^{\ocal_\fcal \leftmod}_*
\cong
[DE \fcal_+, DE \fcal_+ \smashprod \sigma_n]^{DE \fcal_+ \leftmod}_*
\cong
[S^0, DE \fcal_+ \smashprod \sigma_n]^{\torus}_*
\]
Similar isomorphisms also hold for the other two components, so by the calculations
of \cite[Example 5.8.1]{gre99}, we have
\[
\begin{array}{rcl}
\h_*(A)=& \pi_*^\torus(DE \fcal_+ \smashprod \sigma_n) & = \qq_n \langle 1 \rangle \\
\h_*(B)=& \pi_*^\torus(DE \fcal_+ \smashprod \widetilde{E} \fcal \smashprod \sigma_n) & = 0\\
\h_*(C)=& \pi_*^\torus(\widetilde{E} \fcal \smashprod \sigma_n) & = 0 
\end{array}
\]
where $\qq_n \langle 1 \rangle$ is the torsion $\ocal_\fcal$-module consisting of a copy of $\qq$ in factor $n$ and degree 1.
It is immediate that there is a homology isomorphism
\[
\widetilde{\sigma}_n= (\bQ_n\langle 1\rangle \lra 0 \longleftarrow 0)
\longrightarrow 
(A \to B \leftarrow C) = k_n
\]
given by simply picking a suitable representative cycle for $1 \in \qq_n \langle 1 \rangle$.
We therefore have the following description of the cells. 
\begin{lemma}
The set of cells $K_a$ is given (up to weak equivalence) by 
all shifts of objects of the form 
$\widetilde{\sigma}_n$ for $n \geqslant 1$
and all shifts of 
$S^0=(\ocal_\fcal \lra \ecal^{-1} \ocal_\fcal \longleftarrow \qq)$.
\end{lemma}

The above argument on the behaviour of the derived adjunction 
extends to the following useful result, which tells us that 
(after applying homology) our derived functors agree with the functor 
$\pi_*^\acal$ of \cite{gre99}.

\begin{theorem}
Let $X$ be rational $\torus$--equivariant spectrum. 
Let $\Upsilon X$ be its image in $S_a^\bullet \leftmod$.
Then $\h_*(\Upsilon X) \cong l^* \pi_*^\acal(X)$. 
\end{theorem}

The adjunction $(l^*, \Gamma)$ is shown to be
a symmetric monoidal Quillen equivalence between
$d \acal$ with the dualizable model structure and a cellularization of $S_a^\bullet \leftmod$
in \cite[Theorem 7.6]{barnesmonoidal}.
The cells for this cellularization are taken to be
the 'algebraic spheres'.
An algebraic sphere is an object of the form
\[
S^{\nu} = (\ocal_\fcal(\nu) \lra \ecal^{-1} \ocal_\fcal \otimes \qq \longleftarrow \bQ)
\]
where $\ocal_\fcal(\nu)$ is the subset of $\ecal^{-1} \ocal_\fcal$
consisting of all those $x$ such that $c^\nu x \in \ocal_\fcal$,
for $\nu \co \fcal \lra \zz_{\geqslant 0}$ of finite support.
We also allow negative spheres $S^{-\nu}$ and shifts of such objects.
Essentially these are just `partial shifts' of the unit where we have shifted
finitely many factors of $\ocal_\fcal$ by some varying amount.
We let $\{ S^\nu \}$ denote the set of such objects. 

To show that $(l^*,\Gamma)$ is a Quillen equivalence between $d \acal$ with the dualizable model structure and the cellularization of $S_a^\bullet \leftmod$ at the set of cells $K_a$, 
we want to use  \cite[Theorem 7.6]{barnesmonoidal} which says that 
$d \acal$ with the dualizable model structure is Quillen equivalent to the cellularization of $S_a^\bullet \leftmod$ at the set of cells $\{ S^\nu \}$. 
Hence, it is enough to show that these two cellularizations agree 
(that is, produce the same model structure). 
We will prove that the algebraic spheres can be built via cofibre sequences and coproducts in $S_a^\bullet \leftmod$ from cells in $K_a$ and vice-versa. 
It will follow that the class of $K_a$--cellular objects equals the class of 
$\{S^{\nu}\}$--cellular objects. Hence we will see that the 
$K_a$--cellular equivalences and the $\{S^{\nu}\} $--cellular equivalences agree
and that the model categories $K_a \cell S_a^\bullet \leftmod$ and  
$\{S^{\nu}\} \cell S_a^\bullet \leftmod$ are equal. 

The unit $S^0$ (and all its suspensions) is in both sets: in $K_a$ and in the set of `algebraic spheres'. So consider the algebraic sphere $S^{\nu_1}$ for the function $\nu_1: \fcal \lra \zz_{\geqslant 0}$ sending a trivial subgroup to $1$ and all other subgroups to $0$.
There is a cofibre sequence (in $S_a^\bullet \leftmod$):
\[
S^0 \lra S^{\nu_1} \lra \Sigma \sigma_1
\]
where $\Sigma$ denotes the suspension.
This shows that we can build $S^{\nu_1}$ from $\sigma_1$ and  $S^0$ and that we can build 
$\sigma_1$ from algebraic spheres. 
We can also create the negative sphere $S^{-\nu_1}$ using the cofibre sequence:
\[
S^{-\nu_1} \lra S^0 \lra \Sigma^{-1}\sigma_1.
\]
To build any algebraic sphere we apply the above argument repeatedly,  
note that by the definition of an algebraic sphere we need only finitely many steps. 
Equally we can make all $\sigma_i$ for $i \geqslant 1$ from the algebraic spheres.

By \cite[Theorem 7.6]{barnesmonoidal} we have the following. 

\begin{proposition}\label{prop:conclusion}
The pair $(l^*,\Gamma)$ induces a
symmetric monoidal Quillen equivalence between
the model categories
$d \acal_{dual}$ and
$K_a \cell S_a^\bullet \leftmod$.
\end{proposition}

This finishes the proof that $d \acal_{dual}$ provides an algebraic model for the category of rational $\torus$--spectra. We leave the consideration that all our Quillen equivalences are in fact symmetric monoidal to the last section.

\section{Symmetric monoidal equivalences}\label{sec:monoidal}

All of the adjunctions in the zig-zag between
$d \acal_{dual}$ and $\TSP$ have been compatible with the monoidal properties
of the categories.
By examining the cellularized model structures more clearly
we are able to show that each of these model categories
are proper, stable, cellular
monoidal model categories that satisfy
the monoid axiom.
We are thus able to conclude that this zig-zag of Quillen equivalences 
consists of \emph{monoidal} Quillen equivalences. It follows
that we also have Quillen equivalences of model categories of 
ring objects and modules over ring objects.

Our method is to prove a monoidal version of the
cellularization principle, \cite[Proposition 2.7]{gscell},
see Propositions \ref{prop:monoidalleft} and \ref{prop:monoidalright}.

\subsection{Cellularization of stable model categories}\label{subsec:cellularization}

A cellularization of a model category
is a right Bousfield localization at a set of objects.
Such a localization exists by \cite[Theorem 5.1.1]{hir03}
whenever the model category is right proper and cellular.
When we are in a stable context the results of \cite{brstable}
can be used.

Those results, which we shall introduce in the next subsection,
allow us to understand the sets of generating cofibrations
for our cellularized model categories and
see that they are all symmetric monoidal and cellular.

In this subsection we recall the notion of cellularization (when $\ccal$ is stable)
and some of basic definitions and results.

\begin{definition}
Let $\ccal$ be a stable model category and $K$ a stable set of objects of $\ccal$, i.e. the class of $K$--cellular objects of $\ccal$ is closed under desuspension.\footnote{Note that the class is always closed under suspension.}
We say that a map $f \co A \longrightarrow B$ of $\ccal$ is a \textbf{$K$--cellular equivalence} if
the induced map
\[
[k,f]^\ccal_*: [k,A]^\ccal_* \longrightarrow [k,B]^\ccal_*
\]
is an isomorphism of graded abelian groups for each $k \in K$. An object $Z \in \ccal$ is said to be
\textbf{$K$--cellular} if
\[
[Z,f]^\ccal_*: [Z,A]^\ccal_* \longrightarrow [Z,B]^\ccal_*
\]
is an isomorphism of graded abelian groups for any $K$--cellular equivalence $f$.
\end{definition}

\begin{definition}
A \textbf{right Bousfield localization} or \textbf{cellularization} of $\ccal$ with respect to
a set of objects $K$ is a model structure $K \cell \ccal$ on $\ccal$ such that
\begin{itemize}[noitemsep]
\item the weak equivalences are $K$--cellular equivalences
\item the fibrations of $K \cell \ccal$ are the fibrations of $\ccal$
\item the cofibrations of $K \cell \ccal$ are defined via left lifting property.
\end{itemize}
\end{definition}

By \cite[Theorem 5.1.1]{hir03}, if $\ccal$ is a right proper, cellular model category
and $K$ a set of objects in $\ccal$, then the cellularization of $\ccal$ with respect to $K$, $K \cell \ccal$,
exists and is a right proper model category.
The cofibrant objects of $K \cell \ccal$
are called \textbf{$K$--cofibrant} and are precisely the
$K$--cellular and cofibrant objects of $\ccal$.

We recall some definitions and results from \cite{brstable} and prove our monoidal
version of the cellularization principle.

\begin{definition}
Let $K$ be a set of cofibrant objects in a monoidal model category $\ccal$.
We say that $K$ is \textbf{monoidal} if the following two conditions hold.
\begin{itemize}[noitemsep]
\item Any object of the form $k \otimes k'$,
for $k, k' \in K$, is $K$--cellular.
\item For $\cofrep_K S_\ccal$ a $K$--cofibrant replacement of the
unit $S_\ccal$ of $\ccal$ and any $k \in K$, the map
$\cofrep_K S_\ccal \otimes k \to k$ is a $K$--cellular equivalence.
\end{itemize}
\end{definition}
The second condition holds automatically if the unit of $\ccal$ is
$K$--cellular.

The cellularization of a right proper, cellular, stable model category at a stable set of cofibrant objects $K$ is very well behaved (see \cite[Theorem 5.9]{brstable}), in particular it is proper, cellular and stable.

There is another important property we will often want the cells to satisfy,
which makes right localization behave in an even more tractable manner,
see \cite[Section 9]{brstable}. This property is variously called
small, compact or finite. We choose to call it
homotopically small to try and avoid those over-used terms.
\begin{definition}\label{def:hocompact}
We say that an object $X$ of a stable model category $\ccal$
is \textbf{homotopically small}
if in the homotopy category
$[X, \coprod_i Y_i]^\ccal$ is canonically isomorphic to
$\oplus_i [X, Y_i]^\ccal$,
see \cite[Definition 2.1.2]{ss03stabmodcat}.
\end{definition}

Using \cite[Lemma 2.2.1]{ss03stabmodcat}
it is routine to check that if $K$ consists of
homotopically small objects
of $\ccal$ then $K$ is a set of generators for $K \cell \ccal$.
Hence we know a set of generators for each of our cellularizations.

Notice, that derived functors of both left and right Quillen equivalences preserve homotopically small objects. Now we may turn to monoidal considerations. The following theorem is
\cite[Theorem 7.2]{brstable}.

\begin{theorem}
Let $\ccal$ be a proper, monoidal, cellular, stable model category.
Let $K$ be a monoidal and stable set of cofibrant objects of
$\ccal$.
Then $K \cell \ccal$ is a proper, monoidal, cellular, stable model category.
Furthermore, if $\ccal$ satisfies the monoid axiom
then so does $K \cell \ccal$.
\end{theorem}

The next two results are our upgraded version of the cellularization
principle, see  \cite[Proposition 2.7]{gscell}.
They have slightly different assumptions according to whether
the given cells are on the left or right of the adjunction.
The first has the cells on the left and behaves as expected. The second
starts with cells on the right of the adjunction and here we need to assume that
the adjunction is a Quillen equivalence to start with.
In both cases we have also assumed that a cofibrant replacement of the unit is
in the set of cells (and hence is homotopically small). This simplifies the proofs
but is not needed when the adjunction is already a Quillen equivalence.

\begin{proposition}\label{prop:monoidalleft}
Consider a symmetric monoidal Quillen adjunction
between a pair of proper, cellular, stable, monoidal model categories.
\[
\xymatrix{
L : \ccal \ar@<+1ex>[r] & \dcal : R \ar@<+0.5ex>[l] }
\]
Let $K$ be a stable and monoidal set of cofibrant objects of $\ccal$
which contains a cofibrant replacement of the unit.
Assume that each element of $K$ and $LK$ is
homotopically small and that the unit map
$k \to R \fibrep Lk$ is a weak equivalence of $\ccal$ for each $k \in K$.
Then $LK$ is a stable monoidal set of cofibrant objects of
$\dcal$ and the unit of $\dcal$ is in $LK$ (up to weak equivalence).
Moreover, we have an induced symmetric monoidal Quillen equivalence
\[
\xymatrix{
L : K \cell \ccal \ar@<+1ex>[r] & LK \cell \dcal : R. \ar@<+0.5ex>[l] }
\]
\end{proposition}
\begin{proof}
We apply the cellularization principle \cite[Proposition 2.7]{gscell}
to see that $(L,R)$ is a Quillen
equivalence on the cellularized categories.

We must show that $LK$ satisfies both parts of the definition of a monoidal set.
For the first part, let $k$ and $k'$ be objects of $K$. Then $Lk \smashprod Lk'$ is weakly equivalent
to $L(k \smashprod k')$, which is $LK$--cofibrant and hence is $LK$--cellular.
For the second part, the map
$L(\cofrep S_\ccal) \to S_\dcal$ is a weak equivalence since $(L,R)$ is a monoidal
Quillen pair.
Hence $S_\dcal$ is in $LK$ (up to weak equivalence)
and the second condition holds automatically.

Now we know that $LK \cell \dcal$ is an cellular monoidal model category.
We must show that $(L,R)$ is a symmetric monoidal Quillen adjunction
on the cellularized model categories.
We know that the map
$L(\cofrep S_\ccal) \to S_\dcal$ is a weak equivalence.
The comonoidal map $L(X \smashprod Y) \to LX \smashprod LY$
is also a weak equivalence for any cofibrant $X$ and $Y$.
Hence the proof is complete.
\end{proof}

For the following we let $\cofrep$ be the cofibrant replacement functor of  $\ccal$ and
$\fibrep$ the fibrant replacement functor of  $\dcal$.

\begin{proposition}\label{prop:monoidalright}
Consider a symmetric monoidal Quillen equivalence
between a pair of proper, cellular, stable, monoidal model categories
\[
\xymatrix{
L : \ccal \ar@<+1ex>[r] & \dcal : R.  \ar@<+0.5ex>[l]}
\]
Let $H$ be a stable and monoidal set of cofibrant objects of $\dcal$
which contains a cofibrant replacement of the unit of $\dcal$.
Assume that every element of $H$ is homotopically small.
Then $\cofrep R\fibrep H$
is a stable monoidal set of homotopically small cofibrant objects of
$\ccal$ which contains the unit up to weak equivalence.  Furthermore
we have an induced symmetric monoidal Quillen equivalence
\[
\xymatrix{
L : \cofrep R\fibrep H \cell \ccal
\ar@<+1ex>[r]
& H \cell \dcal : R. \ar@<+0.5ex>[l]}
\]
\end{proposition}
\begin{proof}
We apply the cellularization principle \cite[Proposition 2.7]{gscell}
to see that $(L,R)$ is a Quillen equivalence on the cellularized categories.
We must prove that $K = \cofrep R\fibrep H$ is a monoidal set
and that the unit of $\ccal$
is in $K$ (up to weak equivalence).

It is simple to check that
$L$ takes $K$--cellular equivalences between cofibrant objects
to $H$--cellular equivalences.
Now consider the pair of maps below for $k$ and $k'$ elements of $K$.
\[
L\cofrep_K (k \smashprod k')
\overset{Lq}{\longrightarrow} L(k \smashprod k')
\overset{c}{\longrightarrow}
Lk \smashprod Lk'
\]
The map $c$ is the comonoidal map
of $L$ and hence is a weak equivalence as $(L,R)$ is monoidal.
Since the codomain of $c$ is $H$--cellular, so is the domain of $c$.
The map $Lq$ is $L$ applied to a $K$--cellular equivalence between cofibrant objects,
hence it is a $H$--cellular equivalence.
We have shown that $Lq$ is a $H$--cellular equivalence between
$H$--cellular objects of $\dcal$ and thus must be a weak equivalence.
Since $(L,R)$ is a Quillen equivalence before cellularization,
$q$ must be a weak equivalence of $\ccal$. Thus
$k \smashprod k'$ must be $K$--cellular.

To complete the proof that $K$ is monoidal it will suffice to prove that
$S_\ccal$ is $K$--cellular. Thus we now show that the unit of $\ccal$
is in $K$ up to weak equivalence.
Since $(L,R)$ is a symmetric monoidal Quillen pair, the composite map
\[
L \cofrep S_\ccal \lra L S_\ccal \lra S_\dcal \lra \fibrep S_\dcal
\]
is a weak equivalence. Hence the adjoint
$\cofrep S_\ccal \to R \fibrep S_\dcal$ is a weak equivalence.
Thus we see that $\cofrep S_\ccal$ is in $K$ up to weak equivalence.
We have now shown that the set $K$ is monoidal and
that $K \cell \ccal$ is a symmetric monoidal model category.

The proof that this adjunction is symmetric monoidal
on the cellularized model categories
follows the same pattern as the previous case.
\end{proof}

\subsection{Application to the classification}\label{sec:application}

We start with the Quillen equivalence of
Proposition \ref{prop:Smodulesequivalence}.
\[
\xymatrix{
S^\bullet \smashprod - :
L_{S_\bQ}(\TSP)
\ar@<+1ex>[r] &
K_{\Top} \cell S^\bullet \leftmod
: \pullback \ar@<+0.5ex>[l] }
\]
The set of cells $K_{\Top}$ is given by
$S^\bullet \smashprod -$
applied to the set $K$ of generators of
$L_{S_\bQ}(S_\cscr \leftmod)$.
We know that this set is stable and every element is homotopically small and cofibrant.
By the proof of Proposition \ref{prop:monoidalleft}, 
it also follows that $K_{\Top}$ is a monoidal set.
Thus we may apply Proposition \ref{prop:monoidalleft} to see that the adjunction 
$(S^\bullet \smashprod -, \pullback)$ is symmetric monoidal.

We then have a large number of symmetric monoidal Quillen
equivalences relating
$S^\bullet \leftmod$ and
$S_a^\bullet \leftmod$.
Our initial set of cells $K_{\Top}$ is monoidal, stable, contains the unit
and every element is homotopically small.
Hence Propositions \ref{prop:monoidalleft} and \ref{prop:monoidalright}
tell us that
$K_{\Top} \cell S^\bullet \leftmod$ and
$K_a \cell S_a^\bullet \leftmod$
are Quillen equivalent via symmetric monoidal
Quillen equivalences.

\begin{theorem}\label{thm:cyclicconclusion}
The model category of rational $\torus$--spectra, $\TSP$,
is Quillen equivalent to the algebraic model $d \acal _{dual}$.
Furthermore, these Quillen equivalences are all symmetric monoidal.
Hence the homotopy categories of  $\TSP$
and $d \acal _{dual}$ are equivalent as symmetric monoidal categories.
\end{theorem}

\begin{proof}
This now follows by combining Proposition \ref{prop:Smodulesequivalence}, Corollary \ref{cor:removeequivariancecell}, Corollary \ref{cor:gettingtoalgebra}, Section \ref{subsec:ramod}  and Proposition  \ref{prop:conclusion}  with  Propositions \ref{prop:monoidalright} and \ref{prop:monoidalleft}.
\end{proof}

% \bibliography{davebib}
% \bibliographystyle{alpha}

\end{document}